\documentclass[english]{amsart}
\usepackage{a4wide}

\newtheorem{theorem}{Theorem}[section]
\newtheorem{lemma}[theorem]{Lemma}

\newtheorem{proposition}[theorem]{Proposition}
\newtheorem{corollary}[theorem]{Corollary}

\theoremstyle{definition}
\newtheorem{definition}[theorem]{Definition}

\numberwithin{equation}{section}

\theoremstyle{remark}

\newcommand{\abs}[1]{\lvert#1\rvert}

\usepackage{graphicx}

\usepackage[fixlanguage]{babelbib}
\selectbiblanguage{english}

\usepackage[shortlabels]{enumitem}

\usepackage{amsfonts,amssymb,amsbsy,amsmath,mathrsfs,amsthm}

\usepackage{commath}

\usepackage{hyperref}
\usepackage{url}

\hypersetup{pdfpagemode=UseNone, pdfstartview=FitH,
  colorlinks=true,urlcolor=red,linkcolor=blue,citecolor=black,
  pdftitle={Default Title, Modify},pdfauthor={Your Name},
  pdfkeywords={Modify keywords}}

\RequirePackage{babel}

\providecommand{\abs}[1]{ \lvert#1  \rvert}
\providecommand{\norm}[1]{ \lVert#1  \rVert}
\newcommand{\dmu}{\, d \mu}
\newcommand{\dx}{\, d x}

\newcommand{\dt}{\, d t}

\newcommand{\dla}{\, d \lambda}

\def\Xint#1{\mathchoice
   {\XXint\displaystyle\textstyle{#1}}%
   {\XXint\textstyle\scriptstyle{#1}}%
   {\XXint\scriptstyle\scriptscriptstyle{#1}}%
   {\XXint\scriptscriptstyle\scriptscriptstyle{#1}}%
   \!\int}
\def\XXint#1#2#3{{\setbox0=\hbox{$#1{#2#3}{\int}$}
     \vcenter{\hbox{$#2#3$}}\kern-.5\wd0}}

\def\dashint{\Xint-}

\DeclareMathOperator*{\essinf}{ess\,inf}

\usepackage{cite}
\makeatletter
\newcommand{\citecomment}[2][]{\citen{#2}#1\citevar}
\newcommand{\citeone}[1]{\citecomment{#1}}
\newcommand{\citetwo}[2][]{\citecomment[,~#1]{#2}}
\newcommand{\citevar}{\@ifnextchar\bgroup{;~\citeone}{\@ifnextchar[{;~\citetwo}{]}}}
\newcommand{\citefirst}{\@ifnextchar\bgroup{\citeone}{\@ifnextchar[{\citetwo}{]}}}

\makeatother

\begin{document}

\title{Characterizations of parabolic Muckenhoupt classes}

\author{Juha Kinnunen}
\address{Department of Mathematics, Aalto University, P.O. Box 11100, FI-00076 Aalto, Finland}
\email{juha.k.kinnunen@aalto.fi}

\author{Kim Myyryl\"ainen}
\address{Department of Mathematics, Aalto University, P.O. Box 11100, FI-00076 Aalto, Finland}
\email{kim.myyrylainen@aalto.fi}
\thanks{The second author was supported by the Magnus Ehrnrooth Foundation.}

\subjclass[2020]{42B35, 42B37}

\keywords{Parabolic Muckenhoupt weights, parabolic maximal functions, one-sided weights, doubly nonlinear equation}

\begin{abstract} This paper extends and complements the existing theory for the parabolic Muckenhoupt weights motivated by one-sided maximal functions and a doubly nonlinear parabolic partial differential equation of $p$-Laplace type. The main results include characterizations for the limiting parabolic $A_\infty$ and $A_1$ classes by applying an uncentered parabolic maximal function with a time lag. Several parabolic Calder\'on--Zygmund decompositions, covering and chaining arguments appear in the proofs.
\end{abstract}

\maketitle

\section{Introduction}

This paper discusses parabolic Muckenhoupt weights 
\begin{equation*}
\sup_{R \subset \mathbb R^{n+1}}\biggl( \dashint_{R^-(\gamma)} w \biggr) \biggl( \dashint_{R^+(\gamma)} w^{\frac{1}{1-q}} \biggr)^{q-1} < \infty,
\quad 1<q<\infty,
\end{equation*}
where $R^{\pm}(\gamma)$ are space-time rectangles with a time lag $\gamma\ge 0$, see Definition~\ref{def_parrect}.
This class of weights was introduced by Kinnunen and Saari in~\cite{kinnunenSaariMuckenhoupt,kinnunenSaariParabolicWeighted}.
The main results in~\cite{kinnunenSaariMuckenhoupt,kinnunenSaariParabolicWeighted} are characterizations of weighted norm inequalities for the centered forward in time parabolic maximal functions,
self-improving phenomena related to parabolic reverse H\"older inequalities, factorization results and a Coifman--Rochberg type characterization of parabolic $BMO$. 
This paper complements and extends these results. 
Instead of the centered parabolic maximal function in~\cite{kinnunenSaariMuckenhoupt,kinnunenSaariParabolicWeighted}, the corresponding uncentered maximal function gives a more streamlined theory, see Section~\ref{sec:parmaxfct}. 
Observe that the centered and uncentered maximal functions are not comparable in the parabolic case.

There are many characterizations for the standard Muckenhoupt $A_\infty$ condition, but little is known in the  parabolic case.
This paper is an attempt to create the missing theory and several interesting open questions remain.
Theorem \ref{thm:Ainfty} gives new characterizations for the parabolic Muckenhoupt $A_\infty$ class in terms of quantitative absolute continuity with a time lag.
Section~\ref{sec:RHI} gives a new proof for the parabolic reverse H\"older inequality
\cite[Theorem~5.2]{kinnunenSaariParabolicWeighted}.
A complete theory for the parabolic Muckenhoupt $A_1$ class, including factorization and characterization results for the full range of the time lag, is obtained in Section~\ref{sec:A1theory}.
Theorem \ref{translatedRHI} complements~\cite[Proposition 3.4 (iv) and (vii)]{kinnunenSaariParabolicWeighted} and shows that the results are independent of the time lag and the distance between the upper and lower parts of the rectangles.
Several parabolic Calder\'on--Zygmund decompositions, covering and chaining arguments appear in the proofs.

There are two main motivations for the theory of parabolic Muckenhoupt weights.
On the one hand, it is a higher dimensional version of the one-sided Muckenhoupt condition 
\[
\sup_{x \in \mathbb{R}, L > 0} \frac{1}{L} \int_{x-L}^{x} w \biggl( \frac{1}{L} \int_{x}^{x+L} w^{\frac{1}{1-q}} \biggr)^{q-1} < \infty,
\]
with $\gamma=0$,
introduced by Sawyer~\cite{sawyer1986}. 
The one-dimensional theory is well-understood, see Aimar, Forzani and Mart\'in-Reyes~\cite{AFMR1997}, 
Cruz-Uribe, Neugebauer and  Olesen~\cite{CUNO1995}, Mart\'\i n-Reyes~\cite{Martin1993}, Mart\'\i n-Reyes, Ortega Salvador and de la Torre~\cite{MOST1990},  Mart\'\i n-Reyes, Pick and de la Torre~\cite{MRPT1993},  Mart\'\i n-Reyes and de la Torre~\cite{MRT1992,MRT1994}. 
In addition to~\cite{kinnunenSaariMuckenhoupt,kinnunenSaariParabolicWeighted}, several alternative higher dimensional versions have been studied by Berkovits~\cite{berkovits2011}, Forzani, Mart\'{\i}n-Reyes and Ombrosi~\cite{ForzaniMartinreyesOmbrosi2011}, Lerner and Ombrosi~\cite{LO2010} and Ombrosi~\cite{Ombrosi2005}.

On the other hand, parabolic Muckenhoupt weights are related to the doubly nonlinear parabolic equation
\begin{equation}\label{eq.dnle}
\frac{\partial}{\partial t}(\lvert u\rvert^{p-2}u)-\operatorname{div} A(x,t,u,Du)=0,
\end{equation}
where $A$ is a Carath\'eodory function that satisfies the structural conditions 
\[
A(x,t,u,Du) \cdot Du \ge C_0\lvert Du\rvert^p
\quad\text{and}\quad
\lvert A(x,t,u,Du)\rvert \le C_1\lvert Du\rvert^{p-1}
\]
for some positive constants $C_0$ and $C_1$ with $1<p<\infty$. 
In particular, this class of partial differential equations includes the doubly nonlinear $p$-Laplace equation with $A(x,t,u,Du)=|Du|^{p-2}Du$.
In the natural geometry of \eqref{eq.dnle}, we consider space-time rectangles 
where the time variable scales to the power $p$.
Observe that solutions can be scaled, but constants cannot be added. 
If $u(x,t)$ is a solution, so does $u(\lambda x,\lambda^p t)$ with $\lambda>0$.
The main challenge of \eqref{eq.dnle} is the double nonlinearity both in time and space variables.
Trudinger~\cite{Trudinger1968} showed that a scale and location invariant parabolic Harnack's inequality holds true for nonnegative weak solutions to \eqref{eq.dnle} in parabolic rectangles, see also Gianazza and Vespri~\cite{GV2006}, Kinnunen and Kuusi~\cite{kinnunenkuusi} and Vespri~\cite{V1992}. This implies that nonnegative solutions to \eqref{eq.dnle} are parabolic Muckenhoupt weights with $\gamma>0$.
We note that Harnack's inequality is not true with $\gamma=0$ which can be seen from the heat kernel already when $p=2$.
For recent regularity results for the doubly nonlinear equation, we refer to B\"{o}gelein, Duzaar, Kinnunen and Scheven~\cite{bogelein2021b}, 
B\"{o}gelein, Duzaar and Scheven~\cite{bogelein2022},
B\"{o}gelein, Duzaar and Liao~\cite{bogelein2021a},
B\"{o}gelein, Heran, Sch\"{a}tzler and Singer~\cite{bogelein2021c} and Saari~\cite{localtoglobal}.

\section{Definition and properties of parabolic Muckenhoupt weights}

The underlying space throughout is $\mathbb{R}^{n+1}=\{(x,t):x=(x_1,\dots,x_n)\in\mathbb R^n,t\in\mathbb R\}$.
Unless otherwise stated, constants are positive and the dependencies on parameters are indicated in the brackets.
The Lebesgue measure of a measurable subset $A$ of $\mathbb{R}^{n+1}$ is denoted by $\lvert A\rvert$.
A cube $Q$ is a bounded interval in $\mathbb R^n$, with sides parallel to the coordinate axes and equally long, that is,
$Q=Q(x,L)=\{y \in \mathbb R^n: \lvert y_i-x_i\rvert \leq L,\,i=1,\dots,n\}$
with $x\in\mathbb R^n$ and $L>0$. 
The point $x$ is the center of the cube and $L$ is the side length of the cube. 
Instead of Euclidean cubes, we work with the following collection of parabolic rectangles in $\mathbb{R}^{n+1}$.

\begin{definition}\label{def_parrect}
Let $1<p<\infty$, $x\in\mathbb R^n$, $L>0$ and $t \in \mathbb{R}$.
A parabolic rectangle centered at $(x,t)$ with side length $L$ is
\[
R = R(x,t,L) = Q(x,L) \times (t-L^p, t+L^p)
\]
and its upper and lower parts are
\[
R^+(\gamma) = Q(x,L) \times (t+\gamma L^p, t+L^p) 
\quad\text{and}\quad
R^-(\gamma) = Q(x,L) \times (t - L^p, t - \gamma L^p) ,
\]
where $0\leq \gamma < 1$ is the time lag.
\end{definition}

Note that $R^-(\gamma)$ is the reflection of $R^+(\gamma)$ with respect to the time slice $\mathbb{R}^n \times \{t\}$.
The spatial side length of a parabolic rectangle $R$ is denoted by $l_x(R)=L$ and the time length by $l_t(R)=2L^p$.
For short, we write $R^\pm$ for $R^{\pm}(0)$.
The top of a rectangle $R = R(x,t,L)$ is $Q(x,L) \times\{t+L^p\}$
and the bottom is $Q(x,L) \times\{t-L^p\}$.
The $\lambda$-dilate of $R$  with $\lambda>0$ is denoted by $\lambda R = R(x,t,\lambda L)$.

The integral average of $f \in L^1(A)$ in measurable set $A\subset\mathbb{R}^{n+1}$, with $0<|A|<\infty$, is denoted by
\[
f_A = \dashint_A f \dx \dt = \frac{1}{\lvert A\rvert} \int_A f(x,t)\dx\dt .
\]

This section discusses basic properties of parabolic Muckenhoupt weights. 
We begin with the definition of the uncentered parabolic maximal functions.
The differentials $\dx \dt$ in integrals are omitted in the sequel.

\begin{definition}
Let $0\leq\gamma<1$ and $f$ be a locally integrable function. 
The uncentered forward in time and backward in time parabolic maximal functions are defined by
\[
M^{\gamma+}f(x,t) = \sup_{R^-(\gamma)\ni(x,t)} \dashint_{R^+(\gamma)} \lvert f \rvert
\]
and
\[
M^{\gamma-}f(x,t) = \sup_{R^+(\gamma)\ni(x,t)} \dashint_{R^-(\gamma)} \lvert f \rvert.
\]
\end{definition}

A locally integrable nonnegative function $w$ is called a weight.
We give definitions for parabolic Muckenhoupt classes $A^+_q$ and $A^+_1$.

\begin{definition}
Let $1<q<\infty$ and $0\leq\gamma<1$.
A weight $w$ belongs to the parabolic Muckenhoupt class $A^+_q(\gamma)$ if
\[
[w]_{A^+_q(\gamma)} = 
\sup_{R \subset \mathbb R^{n+1}}
\biggl( \dashint_{R^-(\gamma)} w \biggr) \biggl( \dashint_{R^+(\gamma)} w^{\frac{1}{1-q}} \biggr)^{q-1} < \infty ,
\]
where the supremum is taken over all parabolic rectangles $R \subset \mathbb R^{n+1}$.
If the condition above holds with the time axis reversed, then $w \in A^-_q(\gamma)$.
\end{definition}

\begin{definition}
Let $0\leq\gamma<1$.
A weight $w$ belongs to the parabolic Muckenhoupt class $A^+_1(\gamma)$ if
there exists a constant $C=[w]_{A^+_1(\gamma)}>0$ such that
\[
\dashint_{R^-(\gamma)} w \leq C \essinf_{(x,t)\in R^+(\gamma)} w(x,t)
\]
for every parabolic rectangle $R \subset \mathbb R^{n+1}$.
If the condition above holds with the time axis reversed, then $w \in A^-_1(\gamma)$.
\end{definition}

The class $A_1^+(\gamma)$ can be characterized in terms of the parabolic maximal functions.

\begin{proposition}
\label{thm:maximalA1-cond}
Let $0\leq\gamma<1$.
A weight $w$ is in $A_1^+(\gamma)$ if and only if there exists a constant $C$ such that 
\begin{equation}
\label{maximalA_1-cond}
M^{\gamma-}w(x,t) \leq C w(x,t) 
\end{equation}
for almost every $(x,t) \in \mathbb{R}^{n+1}$.
Moreover, we can choose $C = [w]_{A^+_1(\gamma)}$.
The statement also holds for $A_1^-(\gamma)$ with $M^{\gamma+}$.
\end{proposition}

\begin{proof}
Assume that \eqref{maximalA_1-cond} holds. Then
\[
\dashint_{R^-(\gamma)} w \leq M^{\gamma-} w(x,t) \leq C w(x,t) 
\]
for almost every $(x,t) \in R^+(\gamma)$, and thus by taking the essential infimum over every $(x,t) \in R^+(\gamma)$ we have $w\in A_1^+(\gamma)$.

Then assume that $w\in A_1^+(\gamma)$ with the constant $C=[w]_{A^+_1(\gamma)}$.
Let
\[
E = \{ (x,t)\in \mathbb{R}^{n+1}: M^{\gamma-}w(x,t) > C w(x,t) \}
\]
and $(x,t)\in E$.
There exists a parabolic rectangle $R$ such that $(x,t) \in R^+(\gamma)$ and
\begin{equation}
\label{A1_contra}
\dashint_{R^-(\gamma)} w > C w(x,t).
\end{equation}
For every $\varepsilon>0$ there is a rectangle 
$\widetilde{R}$
whose spatial corners and the bottom time coordinate have rational coordinates
such that $(x,t)\in \widetilde{R}^+(\gamma)$, $R^-(\gamma) \subset \widetilde{R}^-(\gamma)$ and $\lvert \widetilde{R}^-(\gamma) \setminus R^-(\gamma) \rvert <\varepsilon$. 
This is possible since the time interval of $R^+(\gamma)$ is open, and thus there is a positive distance between $t$ and the bottom of $R^+(\gamma)$.
Note that $\vert \widetilde{R}^-(\gamma) \rvert = \lvert R^-(\gamma) \rvert + \lvert \widetilde{R}^-(\gamma) \setminus R^-(\gamma) \rvert < \lvert R^-(\gamma) \rvert + \varepsilon $.
By choosing $\varepsilon>0$ small enough, we have
\[
\dashint_{\widetilde{R}^-(\gamma)} w \geq \frac{1}{\lvert R^-(\gamma) \rvert + \varepsilon} \int_{R^-(\gamma)} w > C w(x,t) .
\]
Hence, we may assume that the spatial corner points and the bottom time coordinate of the parabolic rectangles $R$ satisfying \eqref{A1_contra} are rational.
The $A_1^+(\gamma)$ condition and \eqref{A1_contra} imply that
\[
C w(x,t) < \dashint_{R^-(\gamma)} w \leq C \essinf_{(y,s)\in R^+(\gamma)} w(y,s) .
\]
Since $C>0$, we conclude that
\[
w(x,t) < \essinf_{(y,s)\in R^+(\gamma)} w(y,s) .
\]

Let $\{R_i\}_{i\in\mathbb{N}}$ be an enumeration of parabolic rectangles in $\mathbb{R}^{n+1}$ with rational spatial corners and rational bottom time coordinates. Define
\[
E_i = \Bigl\{(x,t)\in R_i^+(\gamma): w(x,t) < \essinf_{(y,s)\in R_i^+(\gamma)} w(y,s) \Bigr\}
\]
for $i\in\mathbb{N}$.
Then $\lvert E_i \rvert =0$ for every $i\in\mathbb{N}$ and the argument above shows that $E\subset\bigcup_{i\in\mathbb{N}} E_i$.
Thus, we have $\lvert E \rvert =0$ and the claim \eqref{maximalA_1-cond} follows with $C=[w]_{A^+_1(\gamma)}$.
\end{proof}

The next lemma tells that the class of parabolic Muckenhoupt weights is closed under taking maximum and minimum.

\begin{lemma}
\label{lem:trunctation}
Let $1\leq q<\infty$, $0\leq\gamma<1$ and
$w, v\in A_q^+(\gamma)$. Then we have $ \max\{w,v\} \in A_q^+(\gamma)$ and $\min\{w,v\} \in A_q^+(\gamma)$.
\end{lemma}

\begin{proof}

Let $u = \max\{w,v\}$.
For $1<q<\infty$,
we have
\begin{align*}
& \biggl( \dashint_{R^-(\gamma)} u \biggr) \biggl(
\dashint_{R^+(\gamma)} u^{\frac{1}{1-q}} \biggr)^{q-1} 
=
\biggl(  \frac{1}{\lvert R^-(\gamma) \rvert} \int_{R^-(\gamma) \cap \{w>v\}} u \biggr) \biggl( \dashint_{R^+(\gamma)} u^{\frac{1}{1-q}} \biggr)^{q-1} \\
&\qquad\qquad +
\biggl(  \frac{1}{\lvert R^-(\gamma) \rvert} \int_{R^-(\gamma) \cap \{w \leq v\}} u \biggr) \biggl( \dashint_{R^+(\gamma)} u^{\frac{1}{1-q}} \biggr)^{q-1}
\\
&\qquad\leq
\biggl( \dashint_{R^-(\gamma)} w \biggr) \biggl( \dashint_{R^+(\gamma)} u^{\frac{1}{1-q}} \biggr)^{q-1}
+
\biggl( \dashint_{R^-(\gamma)} v \biggr) \biggl( \dashint_{R^+(\gamma)} u^{\frac{1}{1-q}} \biggr)^{q-1} 
\\
&\qquad\leq
\biggl( \dashint_{R^-(\gamma)} w \biggr) \biggl( \dashint_{R^+(\gamma)} w^{\frac{1}{1-q}} \biggr)^{q-1}
+
\biggl( \dashint_{R^-(\gamma)} v \biggr) \biggl( \dashint_{R^+(\gamma)} v^{\frac{1}{1-q}} \biggr)^{q-1}
\\
&\qquad\leq
[w]_{A^+_q(\gamma)} + [v]_{A^+_q(\gamma)} .
\end{align*}
On the other hand, for $q=1$ it holds that
\begin{align*}
\dashint_{R^-(\gamma)} u &\leq
\frac{1}{\lvert R^-(\gamma) \rvert} \int_{R^-(\gamma) \cap \{w>v\}} u + \frac{1}{\lvert R^-(\gamma) \rvert} \int_{R^-(\gamma) \cap \{w\leq v\}} u\\
&\leq\dashint_{R^-(\gamma)} w + \dashint_{R^-(\gamma)} v \\
&\leq
[w]_{A^+_q(\gamma)} \essinf_{(x,t)\in R^+(\gamma)} w(x,t) + [v]_{A^+_q(\gamma)} \essinf_{(x,t)\in R^+(\gamma)} v(x,t) \\
&\leq 
\bigl( [w]_{A^+_q(\gamma)} + [v]_{A^+_q(\gamma)} \bigr) \essinf_{(x,t)\in R^+(\gamma)} u(x,t) .
\end{align*}
By taking supremum over all parabolic rectangles $R\subset\mathbb{R}^{n+1}$, we obtain the first claim.
The corresponding claim for the minimum follows similarly.
\end{proof}

Next we discuss a duality property of the parabolic Muckenhoupt weights.
Here $q'=\frac{q}{q-1}$ denotes the conjugate exponent of $q$.

\begin{lemma}
\label{lem:dualityAp}
Let $1<q<\infty$, $0\leq\gamma<1$ and $w$ be a weight. Then $w\in A^+_q(\gamma)$ if and only if $w^{1-q'} \in A^-_{q'}(\gamma)$.
\end{lemma}

\begin{proof}
Assume that $w\in A^+_q(\gamma)$. Then we have
\begin{align*}
&\biggl( \dashint_{R^+(\gamma)} w^{1-q'} \biggr) \biggl( \dashint_{R^-(\gamma)} (w^{1-q'})^{\frac{1}{1-q'}} \biggr)^{q'-1}\\
&\qquad= \biggl( \dashint_{R^-(\gamma)} w \biggr)^\frac{1}{q-1}
\biggl( \dashint_{R^+(\gamma)} w^{\frac{1}{1-q}} \biggr)
\leq
[w]_{A^+_q(\gamma)}^\frac{1}{q-1}
\end{align*}
for every parabolic rectangle $R\subset\mathbb{R}^{n+1}$.
Thus, we have $w^{1-q'} \in A^-_{q'}(\gamma)$.

For the reverse direction, assume that $w^{1-q'} \in A^-_{q'}(\gamma)$.
It holds that
\begin{align*}
&\biggl( \dashint_{R^-(\gamma)} w \biggr) \biggl( \dashint_{R^+(\gamma)} w^{\frac{1}{1-q}} \biggr)^{q-1}\\
&\qquad=
\biggl( \dashint_{R^+(\gamma)} w^{1-q'} \biggr)^{\frac{1}{q'-1}} \biggl( \dashint_{R^-(\gamma)} (w^{1-q'})^{\frac{1}{1-q'}} \biggr) \leq
[w^{1-q'}]_{A^+_{q'}(\gamma)}^\frac{1}{q'-1}
\end{align*}
for every parabolic rectangle $R\subset\mathbb{R}^{n+1}$.
Thus, we have $w\in A^+_q(\gamma)$.
\end{proof}

\section{Time lag and distance between upper and lower parts}

The following theorem asserts that we can change the time lag in the upper and lower parts of parabolic rectangles and also the distance between them.
In particular, the definition of $A_q^+(\gamma)$, $1\leq q<\infty$, does not depend on $0<\gamma<1$.

\begin{theorem}
\label{thm:timelagchange}
Let $0<\gamma\leq\alpha<1$, $q>1$ and $\tau \geq 1$.
Then $w$ belongs to $A^+_q(\gamma)$ if and only if there exists a constant $C=C(n,p,q,\gamma,\alpha,\tau)$ such that
\begin{equation}
\label{translatedRHI}
\biggl( \dashint_{S^-(\alpha)} w \biggr) \biggl( \dashint_{R^+(\alpha)} w^{\frac{1}{1-q}} \biggr)^{q-1} \leq C
\end{equation}
for every parabolic rectangle $R = R(x,t,L) \subset \mathbb R^{n+1}$, where $S^-(\alpha) = R^+(\alpha) - (0, \tau (1+\alpha) L^p)$ is called the translated lower part of $R$.
Moreover, $w$ belongs to $A^+_1(\gamma)$ if and only if there exists a constant $C$ such that
\[
\dashint_{S^-(\alpha)} w \leq C \essinf_{(x,t)\in R^+(\alpha)} w(x,t)
\]
for every parabolic rectangle $R = R(x,t,L) \subset \mathbb R^{n+1}$.
\end{theorem}

Note that for $\tau = 1$ we have $S^-(\alpha) = R^-(\alpha)$. 
We could also consider translated upper parts of parabolic rectangles but this can be included in the translation of lower parts.

\begin{proof}
Assume that $w\in A^+_q(\gamma)$, $q>1$.
Let $R \subset \mathbb R^{n+1}$ be a parabolic rectangle with side length $L$.
Since $R^\pm(\alpha)\subset R^\pm(\gamma)$,
we have
\begin{align*}
\biggl( \dashint_{R^-(\alpha)} w \biggr) \biggl( \dashint_{R^+(\alpha)} w^{\frac{1}{1-q}} \biggr)^{q-1}
\leq
\biggl( \frac{1-\gamma}{1-\alpha} \biggr)^q
\biggl( \dashint_{R^-(\gamma)} w \biggr) \biggl( \dashint_{R^+(\gamma)} w^{\frac{1}{1-q}} \biggr)^{q-1} 
\leq C_1 ,
\end{align*}
where $C_1 = ((1-\gamma) / (1-\alpha) )^q [w]_{A^+_q(\gamma)} $.
This proves the case $\tau=1$.
By H\"older's inequality, it follows that
\[
\dashint_{R^-(\alpha)} w \leq C_1 \biggl( \dashint_{R^+(\alpha)} w^{\frac{1}{1-q}} \biggr)^{1-q}
\leq
C_1 \dashint_{R^+(\alpha)} w .
\]
Choose $N\in\mathbb{N}$ and $0\leq\beta<1$ such that
$\tau = N + \beta$.
Denote 
$S_k^-(\alpha) = R^+(\alpha) - (0, k (1+\alpha) L^p)$ for $k\in\{1,\dots, N\}$.
Then iterating the previous inequality, we get
\begin{align*}
\dashint_{S_N^-(\alpha)} w
\leq
C_1 \dashint_{S_{N-1}^-(\alpha)} w
\leq
C_1^{N-1} \dashint_{S_{1}^-(\alpha)} w
= 
C_1^{N-1} \dashint_{R^-(\alpha)} w .
\end{align*}
This shows the claim whenever $\beta=0$.

If $\beta>0$,
we partition $S^-(\alpha)$ into subrectangles $U^-_{i}(\alpha)$ with spatial side length $\beta^{1/p} L$ and time length $\beta (1-\alpha) L^p$ such that the overlap of $\{U^-_{i}(\alpha)\}_i$ is bounded by $2^{n+1}$. 
This can be done by dividing each spatial edge of $S^-(\alpha)$ into $\lceil \beta^{-1/p} \rceil$ equally long
subintervals with an overlap bounded by $2$, and the time interval of $S^-(\alpha)$ into $\lceil \beta^{-1} \rceil$ equally long subintervals with an overlap bounded by $2$.
We observe that every $U^+_{i}(\alpha)$ is contained in $S_N^-(\alpha)$.
Then it holds that
\begin{align*}
\int_{S^-(\alpha)} w \leq \sum_i \int_{U_i^-(\alpha)} w \leq C_1
\sum_i \int_{U_i^+(\alpha)} w \leq 2^{n+1} C_1 \int_{S_N^-(\alpha)} w .
\end{align*}
Therefore, we have
\begin{align*}
\dashint_{S^-(\alpha)} w \leq 2^{n+1} C_1 \dashint_{S_N^-(\alpha)} w
\leq
2^{n+1} C_1^{N} \dashint_{R^-(\alpha)} w
\leq
2^{n+1} C_1^{\tau} \dashint_{R^-(\alpha)} w .
\end{align*}
We conclude that
\begin{align*}
\biggl( \dashint_{S^-(\alpha)} w \biggr) \biggl( \dashint_{R^+(\alpha)} w^{\frac{1}{1-q}} \biggr)^{q-1}
\leq
2^{n+1} C_1^{\tau} \biggl( \dashint_{R^-(\alpha)} w \biggr) \biggl( \dashint_{R^+(\alpha)} w^{\frac{1}{1-q}} \biggr)^{q-1}
\leq
2^{n+1} C_1^{\tau+1} .
\end{align*}
By letting $q\to1$, we obtain the same conclusion for $A^+_1(\gamma)$.

We prove the other direction.
Let $R_0 \subset \mathbb R^{n+1}$ be an arbitrary parabolic rectangle.
Without loss of generality, we may assume that the center of $R_0$ is the origin.
Let  $m$ be the smallest integer with
\[
m \geq \log_2 \biggl( \frac{\tau(1+\alpha)}{1-\alpha} \biggr) + \frac{1}{p-1} \biggl( 1 + \log_2 \frac{\tau(1+\alpha)}{\gamma} \biggr) + 2 .
\]
Then there exists $0 \leq \varepsilon < 1$ such that
\[ 
m = \log_2 \biggl( \frac{\tau(1+\alpha)}{1-\alpha} \biggr) +  \frac{1}{p-1} \biggl( 1 + \log_2 \frac{\tau(1+\alpha)}{\gamma} \biggr) + 2 + \varepsilon .
\]
We partition $R^+_0(\gamma) = Q(0, L) \times (\gamma L^p, L^p)$ by dividing each of its spatial edges into $2^m$ equally long intervals and the time interval into 
$\lceil (1-\gamma)2^{mp}/(1-\alpha)\rceil$ equally long intervals.
Denote the obtained rectangles by $U^+_{i,j}$ with $i \in \{1,\dots,2^{mn}\}$ and  $j \in \{1,\dots,\lceil (1-\gamma)2^{mp}/(1-\alpha)\rceil\}$.
The spatial side length of $U^+_{i,j}$ is $l = l_x(U^+_{i,j}) =L/2^m$
and the time length is
\[
l_t(U^+_{i,j}) = \frac{(1-\gamma)L^p}{\lceil (1-\gamma)2^{mp}/(1-\alpha)\rceil} .
\]
For every $U^+_{i,j}$, there exists a unique rectangle $R_{i,j}^+(\alpha)$ that has the same top as $U^+_{i,j}$.
Our aim is to construct a chain from each $U^+_{i,j}$ to a central rectangle which is of the same form as $R_{i,j}^+(\alpha)$ and is contained in $R_0^+$. 
This central rectangle will be specified later.
First, we construct a chain with respect to the spatial variable.
Fix $U^+_{i,j}$.
Let 
\[
P_0^+ = R^+_{i,j}(\alpha) = Q_i \times (t_j - (1-\alpha)l^p, t_j)
\]
and
\[
P_0^- = S^-_{i,j}(\alpha) = R^+_{i,j}(\alpha) - (0, \tau (1+\alpha) l^p) .
\]
Here $S^-_{i,j}(\alpha)$ is the translated lower part of $R_{i,j}$.
We construct a chain of cubes from $Q_i$ to the central cube $Q(0,l)$.
Let $Q'_0 = Q_i = Q(x_i, l)$ and set
\[
Q'_k = Q'_{k-1} - \frac{x_i}{\abs{x_i}} \frac{\theta l}{2}, 
\quad k \in \{0,\dots,N_i\} ,
\]
where $1 \leq \theta \leq \sqrt{n}$ depends on the angle between $x_i$ and the spatial axes and is chosen such that the center of $Q_k$ is on the boundary of $Q_{k-1}$.
We have
\[
\frac{1}{2^n} \leq \frac{\abs{ Q_k \cap Q_{k-1} }}{\abs{Q_k}} \leq \frac{1}{2}, 
\quad k \in \{0,\dots,N_i\} ,
\]
and $\abs{x_i} = \frac{\theta}{2} (L - bl)$,
where $b \in \{1, \dots, 2^m\}$ depends on the distance of $Q_i$ to the center of $Q_0 = Q(0,L)$.
The number of cubes in the spatial chain $\{Q'_k\}_{k=0}^{N_i}$ is
\[
N_i + 1 = \frac{\abs{x_i}}{\frac{\theta}{2} l} + 1 = \frac{L}{l} - b + 1 .
\]

Next, we also take the time variable into consideration in the construction of the chain.
Let
\[
P^+_k = Q'_k \times ( t_j - (1-\alpha)l^p - k\tau(1+\alpha)l^p, t_j - k\tau(1+\alpha)l^p )
\]
and $P^-_k = P^+_k - (0, \tau (1+\alpha)l^p)$,
for $k \in \{ 0, \dots, N_i \}$,
be the upper and the translated lower parts of a parabolic rectangle respectively.
These will form a chain of parabolic rectangles from $U^+_{i,j}$ to the eventual central rectangle.
Observe that every rectangle $P_{N_i}$ coincides spatially for all pairs $(i,j)$.
Consider $j=1$ and such $i$ that the boundary of $Q_i$ intersects the boundary of $Q_0$.
For such a cube $Q_i$, we have $b=1$, and thus $N = N_i = \frac{L}{l} - 1$.
In the time variable, we travel from $t_1$ the distance
\[
(N+1)\tau(1+\alpha)l^p + (1-\alpha)l^p = \tau (1+\alpha) L l^{p-1} + (1-\alpha) l^p .
\]
We show that the translated lower part of the final rectangle $P^-_N$ is contained in $R_0^+$.
To this end, we subtract the time length of $U^+_{i,1}$ from the distance above and observe that it is less than the time length of $R_0 \setminus R_0^+(\gamma)$.
This follows from the computation
\begin{align*}
&\tau (1+\alpha) L l^{p-1} + (1-\alpha) l^p - \frac{(1-\gamma)L^p}{\lceil(1-\gamma)2^{mp}/(1-\alpha)\rceil}\\
&\qquad= \biggl( \frac{\tau(1+\alpha)}{2^{m(p-1)}} + \frac{1-\alpha}{2^{mp}} - \frac{1-\gamma}{\lceil (1-\gamma)2^{mp}/(1-\alpha) \rceil} \biggr) L^p \\
&\qquad\leq \biggl( \frac{\tau(1+\alpha)}{2^{m(p-1)}} + \frac{1-\alpha}{2^{mp}} - \frac{1-\gamma}{2\frac{(1-\gamma)2^{mp}}{1-\alpha}} \biggr) L^p  \\
&\qquad= \biggl( \frac{\tau(1+\alpha)}{2^{m(p-1)}} + \frac{1-\alpha}{2^{mp+1}} \biggr) L^p 
\leq 2 \frac{\tau(1+\alpha)}{2^{m(p-1)}} L^p \leq \gamma L^p ,
\end{align*}
since 
\[
m \geq \frac{1}{p-1} \biggl( 1 + \log_2 \frac{\tau(1+\alpha)}{\gamma} \biggr) .
\]
This implies that $P^-_N \subset R_0^+$.
Denote these rectangles $P_N^+$ and $P_N^-$ by $\mathfrak{R}^+$ and $\mathfrak{R}^-$, respectively.
These are the central rectangles where all chains will eventually end.

Let $j=1$ and assume that $i$ is arbitrary. We extend the chain $\{P_k^+\}_{k=0}^{N_i}$ by $N - N_i$ rectangles into the negative time direction such that the final rectangle coincides with the central rectangle $\mathfrak{R}^+$.
More precisely, we consider $Q'_{k+1} = Q'_{N_i}$, 
\[
P^+_{k+1} = P^+_{k} - (0, \tau (1+\alpha)l^p) 
\quad\text{and}\quad
P^-_{k+1} = P^+_{k+1} - (0, \tau (1+\alpha)l^p)
\]
for $k \in \{N_i, \dots, N-1\}$. 
For every $j \in \{2,\dots,\lceil (1-\gamma)2^{mp}/(1-\alpha)\rceil \}$, we consider a similar extension of the chain.
The final rectangles of the chains coincide for fixed $j$ and for every $i$.
Moreover, every chain is of the same length $N+1$, and it holds that
\[
\frac{1}{2^n} \leq \eta = \frac{\lvert P_k^+ \cap P_{k-1}^- \rvert}{\lvert P_k^+ \rvert} \leq 1 .
\]

Then we consider an index $j \in \{2,\dots,\lceil(1-\gamma)2^{mp}/(1-\alpha)\rceil \}$ related to the time variable.
The time distance between the current ends of the chains for pairs $(i,j)$ and $(i,1)$ is
\[
(j-1) \frac{(1-\gamma) L^p}{\lceil (1-\gamma)2^{mp}/(1-\alpha)\rceil} .
\]
Our objective is to have the final rectangle of the continued chain for $(i,j)$ to coincide with the end of the chain for $(i,1)$, that is, with the central rectangle $\mathfrak{R}^+$.
To achieve this, we modify $2^{m-1}$ intersections of $P^+_k$ and $P^-_{k+1}$ by shifting $P_k^+$ and also add a chain of $M_j$ rectangles traveling to the negative time direction into the chain $\{P_k^+\}_{k=0}^{N}$.
We shift every $P_k^+$ and $P_k^-$, $k \in \{ 1, \dots, 2^{m-1} \}$, by a $\beta_j$-portion of their temporal length more than the previous rectangle was shifted, that is, we move each $P_k^+$ and $P_k^-$ into the negative time direction a distance of $k \beta_j (1-\alpha) l^p$.
The values of $M_j \in \mathbb{N}$ and $0\leq \beta_j <1$ will be chosen later.
In other words, modify the definitions of $P^+_k$ for $ k \in \{ 1, \dots, 2^{m-1} \}$ by
\[
P^+_k = Q'_k \times ( t_j - (1-\alpha)l^p - k( \tau(1+\alpha) + \beta_j (1-\alpha)) l^p, t_j - k( \tau(1+\alpha) + \beta_j (1-\alpha)) l^p) ,
\]
and then add $M_j$ rectangles defined by
\[
P^+_{k+1} = P^+_{k} - (0, \tau(1+\alpha)l^p) 
\quad\text{and}\quad
P^-_{k+1} = P^+_{k+1} - (0, \tau(1+\alpha)l^p) 
\]
for $k \in \{N,\dots, N + M_j-1\}$.
Note that there exists $1 \leq \omega < 2$ such that 
\[
\omega \frac{(1-\gamma)2^{mp}}{1-\alpha} 
=\biggl\lceil\frac{(1-\gamma)2^{mp}}{(1-\alpha)}\biggr\rceil.
\]
We would like to find such $0\leq \beta_j <1$ and $M_j \in \mathbb{N}$ that 
\[
(j-1) \frac{(1-\gamma) L^p}{\lceil (1-\gamma)2^{mp}/(1-\alpha)\rceil} - M_j \frac{\tau(1+\alpha) L^p}{2^{mp}} 
= 2^{m-1} \beta_j \frac{(1-\alpha) L^p}{2^{mp}} ,
\]
which is equivalent with
\[
(j-1)\omega^{-1} (1-\alpha) - M_j \tau (1+\alpha) = 2^{m-1} \beta_j (1-\alpha) .
\]
With this choice all final rectangles coincide.
Choose $M_j \in \mathbb{N}$ such that
\[
M_j \tau (1+\alpha) \leq (j-1) \omega^{-1} (1-\alpha) < (M_j + 1)\tau(1+\alpha) ,
\]
that is,
\[
0 \leq \xi = (j-1) \omega^{-1} (1-\alpha) - M_j \tau (1+\alpha) < \tau(1+\alpha)
\]
and
\begin{align*}
\frac{(j-1)(1-\alpha)}{2 \tau (1+\alpha)} -1 
&\leq \frac{(j-1)(1-\alpha)}{\omega \tau (1+\alpha)} -1 
< M_j \leq \frac{(j-1)(1-\alpha)}{\omega \tau (1+\alpha)} \\
&\leq \frac{(j-1)(1-\alpha)}{\tau(1+\alpha)} .
\end{align*}
By choosing $0\leq \beta_j <1$ such that
\[
\xi = 2^{m-1} \beta_j (1-\alpha) = 2^{\frac{1}{p-1} + 1 + \varepsilon} \biggl( \frac{1+\alpha}{\gamma} \biggr)^\frac{1}{p-1} \beta_j \tau (1+\alpha) ,
\]
we have
\[
\beta_j = 2^{-\frac{1}{p-1} - 1 - \varepsilon} \biggl( \frac{\gamma}{1+\alpha} \biggr)^\frac{1}{p-1} \frac{\xi}{\tau(1+\alpha)} .
\]
Observe that $0\leq \beta_j \leq \frac{1}{2}$ for every $j$.
For measures of the intersections of the modified rectangles, it holds that
\[
\frac{1}{2^{n+1}} \leq \frac{\lvert P_k^+ \cap P_{k-1}^- \rvert}{\lvert P_k^+ \rvert} = \eta (1-\beta_j) \leq 1 
\]
for $ k \in \{ 1, \dots, 2^{m-1} \}$, and thus
\[
\frac{1}{2^{n+1}} \leq \tilde{\eta}_j = \frac{\lvert P_k^+ \cap P_{k-1}^- \rvert}{\lvert P_k^+ \rvert} \leq 1 
\]
for every $k \in \{1,\dots, N + M_j\}$.
Fix $U^+_{i,j}$.
Denote $\delta = 1/(q-1)$.
H\"older's inequality and the assumption \eqref{translatedRHI} imply that
\begin{align*}
\biggl( \dashint_{P_{k}^+} w^{-\delta} \biggr)^{-1/\delta} &\leq \Biggl( \frac{\lvert P_{k}^+ \cap P_{k-1}^- \rvert }{\lvert P_{k}^+ \rvert} \dashint_{P_{k}^+ \cap P_{k-1}^-} w^{-\delta} \Biggr)^{-1/\delta} \leq 2^{\frac{n+1}{\delta}} \biggl( \dashint_{P_{k}^+ \cap P_{k-1}^-} w^{-\delta} \biggr)^{-1/\delta} \\
&\leq 2^{\frac{n+1}{\delta}} \dashint_{P_{k}^+ \cap P_{k-1}^-} w \leq 2^{\frac{n+1}{\delta}} \frac{\lvert P_{k-1}^- \rvert}{\lvert P_{k}^+ \cap P_{k-1}^- \rvert} \dashint_{P_{k-1}^-} w\\
&\leq 2^{(n+1)(1+\frac{1}{\delta})} \dashint_{P_{k-1}^-} w 
\leq c_0 \biggl( \dashint_{P_{k-1}^+} w^{-\delta} \biggr)^{-1/\delta}
\end{align*}
for every $k \in \{1,\dots, N + M_j\}$, where $c_0 = 2^{(n+1)(1+\frac{1}{\delta})} C = 2^{(n+1)q} C$.
By iterating the inequality above, we obtain
\begin{align*}
\biggl( \dashint_{\mathfrak R^+} w^{-\delta} \biggr)^{-1/\delta} &= \biggl( \dashint_{P_{N+M_j}^+} w^{-\delta} \biggr)^{-1/\delta}
\leq c_0^{N+M_j} \biggl( \dashint_{P_{0}^+} w^{-\delta} \biggr)^{-1/\delta} \\
&\leq c_1 \biggl( \dashint_{R_{i,j}^+(\alpha)} w^{-\delta} \biggr)^{-1/\delta} ,
\end{align*}
where $c_1 = c_0^{s-1} $
and
\begin{align*}
N+1+M_j &= \frac{L}{l}+M_j \leq 2^m + (j-1) \frac{1-\alpha}{\tau(1+\alpha)} \\
&\leq 2^m + \frac{(1-\gamma)2^{mp}}{1-\alpha} \frac{1-\alpha}{\tau(1+\alpha)} 
\leq 2^m + \frac{2^{mp}}{\tau} \leq 2^{mp+1} \\
&\leq 2^{\frac{p}{p-1} + 3p +1} \biggl( \frac{\tau(1+\alpha)}{\gamma} \biggr)^\frac{p}{p-1} \biggl( \frac{\tau(1+\alpha)}{1-\alpha} \biggr)^p = s
\end{align*}
for every $j$.
Recall that $i \in \{1,\dots,2^{mn}\}$ and  $j \in \{1,\dots,\lceil (1-\gamma)2^{mp}/(1-\alpha)\rceil\}$.
We observe that
\begin{align*}
\dashint_{R^+_0(\gamma)} w^{-\delta} &= \sum_{i,j} \frac{\lvert U^+_{i,j} \rvert}{\lvert R^+_0(\gamma) \rvert} \dashint_{U^+_{i,j}} w^{-\delta} \leq \frac{1}{2^{mn} \Bigl\lceil\frac{(1-\gamma)2^{mp}}{(1-\alpha)}\Bigr\rceil} \sum_{i,j}  \dashint_{R_{i,j}^+(\alpha)} w^{-\delta} \\
&\leq \frac{1}{2^{mn} \Bigl\lceil\frac{(1-\gamma)2^{mp}}{(1-\alpha)}\Bigr\rceil} \sum_{i,j} c_1^{\delta}  \dashint_{\mathfrak R^+} w^{-\delta} \\
&= \frac{1}{2^{mn} \Bigl\lceil\frac{(1-\gamma)2^{mp}}{(1-\alpha)}\Bigr\rceil} 2^{mn} \biggl\lceil\frac{(1-\gamma)2^{mp}}{(1-\alpha)}\biggr\rceil c_1^{\delta} \dashint_{\mathfrak R^+} w^{-\delta} \\
&= c_1^{\delta}  \dashint_{\mathfrak R^+} w^{-\delta}.
\end{align*}
Thus, we have
\begin{equation}
\label{upperRHI}
\biggl( \dashint_{\mathfrak R^+} w^{-\delta} \biggr)^{-1/\delta} \leq c_1 \biggl( \dashint_{R^+_0(\gamma)} w^{-\delta} \biggr)^{-1/\delta} .
\end{equation}

We can apply a similar chaining argument in the reverse time direction for $R_0^-(\gamma)$ with the exception that we also extend (and modify if needed) every chain such that the corresponding central rectangles coincide with $\mathfrak{R}^+$ and $\mathfrak{R}^-$.
A rough upper bound for the number of rectangles needed for the additional extension is given by
\[
\biggl\lceil \frac{2\gamma L^p}{\tau(1+\alpha)l^p} \biggr\rceil = \biggl\lceil \frac{2\gamma }{\tau(1+\alpha)} 2^{mp} \biggr\rceil \leq 2^{mp+1} .
\]
Thus, the constant $s$ above is two times larger in this case. Then again by iterating H\"older's inequality and the assumption \eqref{translatedRHI}, we obtain
\begin{equation}
\label{lowerRHi}
\dashint_{R^-_0(\gamma)} w \leq c_2 \dashint_{\mathfrak R^-} w ,
\end{equation}
where $c_2 = c_0^{2s-1}$.
By combining \eqref{upperRHI} and \eqref{lowerRHi}, we conclude that
\begin{align*}
\dashint_{R^-_0(\gamma)} w \leq c_2 \dashint_{\mathfrak R^-} w \leq c_2 C \biggl( \dashint_{\mathfrak R^+} w^{-\delta} \biggr)^{-1/\delta} \leq c \biggl( \dashint_{R^+_0(\gamma)} w^{-\delta} \biggr)^{-1/\delta}, 
\end{align*}
where $c = c_1 c_2 C = c_0^{3s-2} C 
= 2^{(n+1)q(3s-2)} C^{3s-1}$.
Since $R_0$ was an arbitrary parabolic rectangle in $\mathbb R^{n+1}$, it holds that $w \in A^+_q(\gamma)$. 
This completes the proof for $q>1$. Letting $q\to1$ in the argument above, we obtain the claim for $q=1$.
\end{proof}

\section{Characterizations of parabolic $A_\infty$}

This section discusses several characterizations of the parabolic Muckenhoupt $A_\infty$ condition
in terms of quantitative and qualitative measure conditions. A connection to a parabolic Gurov--Reshetnyak condition is also included.
See~\cite{Korenovskii2007,reshetnyak1994} for the Gurov--Reshetnyak class.

\begin{theorem}
\label{thm:Ainfty}
Let $0<\gamma<1$ and $w$ be a weight.
The following conditions are equivalent.
\begin{enumerate}[(i),topsep=5pt,itemsep=5pt]
\item $w\in A^+_q(\gamma)$ for some $1<q<\infty$.

\item There exist constants $K,\delta >0$ such that
\[
\frac{\lvert E \rvert}{\lvert R^+(\gamma) \rvert} \leq K \biggl( \frac{w(E)}{w(R^-(\gamma))} \biggr)^\delta
\]
for every parabolic rectangle $R\subset\mathbb{R}^{n+1}$ and 
measurable set $E \subset R^+(\gamma)$.

\item For every $0<\alpha<1$ there exists $0<\beta<1$ such that for every parabolic rectangle $R$ and for every measurable set $E \subset R^+(\gamma)$ for which $w(E) < \beta w(R^-(\gamma))$ we have $\lvert E \rvert < \alpha \lvert R^+(\gamma) \rvert$.

\item There exist $0<\alpha,\beta<1$ such that for every parabolic rectangle $R$ and for every measurable set $E \subset R^+(\gamma)$ for which $w(E) < \beta w(R^-(\gamma))$ we have $\lvert E \rvert < \alpha \lvert R^+(\gamma) \rvert$.

\item There exist $0<\alpha,\beta <1$ such that 
for every parabolic rectangle $R$ we have
\[
\lvert R^+(\gamma) \cap \{ w < \beta w_{R^-(\gamma)} \} \rvert < \alpha \lvert R^+(\gamma) \rvert .
\]

\item There exists $0 < \varepsilon < 1$ such that 
for every parabolic rectangle $R$ we have
\[
\dashint_{R^+(\gamma)} (w - w_{R^-(\gamma)})^- \leq \varepsilon w_{R^-(\gamma)} .
\]
\end{enumerate}
\end{theorem}

The proof is presented in subsections below.

\subsection{Quantitative measure condition}

We show that $(i)\Leftrightarrow (ii)$ in Theorem~\ref{thm:Ainfty}.
The following theorem also holds in the case $p=1$.

\begin{theorem}
Let $0\leq \gamma<1$ and $w$ be a weight. 
Then $w\in A^+_q(\gamma)$ for some $1<q<\infty$ if and only if there exist constants $K,\delta >0$ such that
\[
\frac{\lvert E \rvert}{\lvert R^+(\gamma) \rvert} \leq K \biggl( \frac{w(E)}{w(R^-(\gamma))} \biggr)^\delta 
\]
for every parabolic rectangle $R\subset\mathbb{R}^{n+1}$ and 
measurable set $E \subset R^+(\gamma)$.
\end{theorem}

\begin{proof}
Assume first that $w\in A^+_q(\gamma)$.
Let $E$ be a measurable subset of $R^+(\gamma)$.
By H\"older's inequality, we have
\begin{align*}
\frac{\lvert E \rvert}{\lvert R^+(\gamma) \rvert} &= \dashint_{R^+(\gamma)} \chi_{E} = \frac{1}{\lvert R^+(\gamma) \vert} \int_{R^+(\gamma)} w^{-\frac{1}{q}} w^{\frac{1}{q}} \chi_{E} \\
&\leq \frac{1}{\lvert R^+(\gamma) \vert} \biggl( \int_{R^+(\gamma)} w^{\frac{1}{1-q}} \biggr)^{\frac{q-1}{q}} \biggl( \int_{R^+(\gamma)} w \chi_{E} \biggr)^\frac{1}{q} \\
&= \frac{1}{\lvert R^+(\gamma) \vert^\frac{1}{q}} \biggl( \dashint_{R^+(\gamma)} w^{\frac{1}{1-q}} \biggr)^{\frac{q-1}{q}} \biggl( \int_{E} w \biggr)^\frac{1}{q} \\
&= \Biggl( \frac{w(R^-(\gamma))}{\lvert R^-(\gamma) \vert} \biggl( \dashint_{R^+(\gamma)} w^{\frac{1}{1-q}} \biggr)^{q-1} \Biggr)^\frac{1}{q} \biggl( \frac{w(E)}{w(R^-(\gamma))} \biggr)^\frac{1}{q} \\
&\leq [w]_{A^+_q(\gamma)}^\frac{1}{q} \biggl( \frac{w(E)}{w(R^-(\gamma))} \biggr)^\frac{1}{q} .
\end{align*}

Then we prove the other direction.
The assumption is equivalent with
\[
\frac{w(R^-(\gamma))}{w(E)} \leq K^q \biggl( \frac{\lvert R^+(\gamma) \rvert}{\lvert E \rvert} \biggr)^q ,
\]
where $K >0$, $q=\delta^{-1}>0$ and $E$ is a measurable subset of $R^+(\gamma)$.
Since the ratio of the Lebesgue measure of $R^+(\gamma)$ to the Lebesgue measure of $E$ is always greater than or equal to 1, we may assume without loss of generality that the exponent $q$ is strictly greater than 1.
Denote $E_\lambda = R^+(\gamma) \cap \{ w^{-1} > \lambda \}$.
We have $w(E_\lambda) \leq \lvert E_\lambda \rvert / \lambda$.
It follows that
\[
\frac{1}{K^q} w(R^-(\gamma)) \biggl( \frac{\lvert E_\lambda \rvert}{\lvert R^+(\gamma) \rvert} \biggr)^q \leq w(E_\lambda) \leq \frac{1}{\lambda} \lvert E_\lambda \rvert ,
\]
and hence we get
\[
\lvert E_\lambda \rvert \leq \frac{K^{q'}}{\lambda^{q'-1}} \frac{\lvert R^-(\gamma) \rvert^{q'}}{w(R^-(\gamma))^{q'-1}} ,
\]
where $q'=\frac{q}{q-1}$ is the conjugate exponent of $q$.
Letting
$0<\varepsilon<q'-1$ and applying Cavalieri's principle allows us to evaluate
\begin{align*}
\int_{R^+(\gamma)} w^{-\varepsilon} &= \varepsilon \int_0^\infty \lambda^{\varepsilon -1} \lvert R^+(\gamma) \cap \{ w^{-1} > \lambda \} \rvert \dla \\
&=
\varepsilon \int_0^{1/w_{R^-(\gamma)}} \lambda^{\varepsilon -1} \lvert E_\lambda \rvert \dla + \varepsilon \int_{ 1/w_{R^-(\gamma)} }^\infty \lambda^{\varepsilon -1} \lvert E_\lambda \rvert \dla \\
&\leq
\lvert R^+(\gamma) \rvert \biggl( \frac{\lvert R^-(\gamma) \rvert}{w(R^-(\gamma))} \biggr)^{\varepsilon} + \varepsilon K^{q'} \frac{\lvert R^- \rvert^{q'}}{w(R^-(\gamma))^{q'-1}}  \int_{1/w_{R^-(\gamma)}}^\infty \lambda^{\varepsilon-q'} \dla \\
&=
\lvert R^+(\gamma) \vert \biggl( \frac{\lvert R^-(\gamma) \rvert}{w(R^-(\gamma))} \biggr)^{\varepsilon} + \frac{\varepsilon K^{q'}}{q'-1-\varepsilon} \frac{\lvert R^-(\gamma) \rvert^{q'}}{w(R^-(\gamma))^{q'-1}} \biggl( \frac{\lvert R^-(\gamma) \rvert}{w(R^-(\gamma))} \biggr)^{\varepsilon-q'+1} \\
&= \biggl( 1 + \frac{\varepsilon K^{q'}}{q'-1-\varepsilon} \biggr) \lvert R^+(\gamma) \vert \biggl( \frac{\lvert R^-(\gamma) \rvert}{w(R^-(\gamma))} \biggr)^{\varepsilon} .
\end{align*}
Thus, we obtain
\[
\frac{w(R^-(\gamma))}{\lvert R^-(\gamma) \rvert} \biggl( \dashint_{R^+(\gamma)} w^{-\varepsilon} \biggr)^\frac{1}{\varepsilon} \leq c ,
\]
where $c^{\varepsilon} = 1 + \varepsilon K^{q'} / (q'-1-\varepsilon)$.
By taking the supremum over all parabolic rectangles, 
we conclude that $w\in A^+_{1+1/\varepsilon}$ 
and thus the proof is complete.
\end{proof}

\subsection{Qualitative measure condition}

We show that $(i)\Leftrightarrow (iv)$ in Theorem~\ref{thm:Ainfty}.

First we note that Theorem~\ref{thm:Ainfty}~$(ii)$ implies $(iii)$, 
since if $w(E) < \beta w(R^-(\gamma))$, then
\begin{align*}
\lvert E \rvert\leq K \biggl( \frac{w(E)}{w(R^-(\gamma))} \biggr)^\delta \lvert R^+(\gamma) \rvert \leq K \beta^\delta \lvert R^+(\gamma) \rvert  ,
\end{align*}
where we can choose $\beta$ small enough such that $K\beta^\delta\leq\alpha$.
The implication from $(iii)$ to $(iv)$ is immediate.

To prove the reverse implication from $(iv)$ to $(i)$,
we need the following auxiliary result.

\begin{lemma}
\label{lemma:doublingforward}
Let $0<\gamma,\alpha,\beta <1$.
Assume that $w$ satisfies the qualitative measure condition, that is, for every parabolic rectangle $R$ and for every measurable set $E \subset R^+(\gamma)$ for which $w(E) < \beta w(R^-(\gamma))$ it holds that $\lvert E \rvert < \alpha \lvert R^+(\gamma) \rvert$. Then we have the following properties.
\begin{enumerate}[(i)]
\item 
For every parabolic rectangle $R$ and every measurable set $E \subset R^+(\gamma)$ for which $\lvert E \rvert \geq \alpha \lvert R^+(\gamma) \rvert$ it holds that $w(E) \geq \beta w(R^-(\gamma))$.
\item
Let $\theta>0$.
For every parabolic rectangle $R$ and $0\leq\omega\leq \theta $
it holds that
\[
w(R^-(\gamma)) \leq C w(R^-(\gamma) + (0, \omega L^p)),
\]
where $C\geq1$ depends on $p,\gamma,\beta$ and $\theta $.
\end{enumerate}

\end{lemma}

\begin{proof}

$(i)$ This is simply the contraposition of the qualitative measure condition.

$(ii)$
Let $\theta>0$ and $R \subset \mathbb{R}^{n+1}$ be a fixed parabolic rectangle of side length $L$.
Choose $m \in \mathbb{N}$ such that
\begin{equation}
\label{partitionbound}
\frac{(1+\gamma) L^p}{2^{pm}} \leq 
\frac{(1-\gamma) L^p}{2}
< \frac{(1+\gamma) L^p}{2^{p(m-1)}} .
\end{equation}
We partition $R^-(\gamma)$ into subrectangles $R^-_{0,i}(\gamma)$ with spatial side length $L /2^m$ and time length $(1-\gamma)L^p /2^{pm}$ such that the overlap of $\{R^-_{0,i}(\gamma)\}_i$ is bounded by $2$. This can be done by dividing each spatial edge of $R^-(\gamma)$ into $2^m$ equally long pairwise disjoint intervals, and the time interval of $R^-(\gamma)$ into $\lceil 2^{pm} \rceil$ equally long subintervals such that their overlap is bounded by $2$.

Our plan is to shift every rectangle $R^-_{0,i}(\gamma)$ forward in time by multiple times of $(1+\gamma) L^p / 2^{pm}$ until the shifted rectangles are contained in $R^-(\gamma) + (0, \theta L^p )$.
To this end, choose
$N \in \mathbb{N}$ such that
\[
(N-1) \frac{(1+\gamma) L^p}{2^{pm}}  < \theta L^p \leq N \frac{(1+\gamma) L^p}{2^{pm}} .
\]
We first move every rectangle $R^-_{0,i}(\gamma)$ forward in time by $(N-1) (1+\gamma) L^p / 2^{pm}$. Then we shift once more by the distance $(1+\gamma) L^p / 2^{pm}$ those rectangles that are not yet subsets of $R^-(\gamma) + (0, \theta L^p ) $.
Denote so obtained shifted rectangles by $R^-_{N,i}(\gamma)$. 
Observe that the choice of $N$ and \eqref{partitionbound} ensures that all shifted rectangles $R^-_{N,i}(\gamma)$ are contained in $R^-(\gamma) + (0, \theta L^p )$.
By the construction and the bounded overlap of $R^-_{0,i}(\gamma)$, the overlap of $R^-_{N,i}(\gamma)$ is bounded by $4$.
Then we apply $(i)$ for $E=R^+_{0,i}(\gamma)$
and continue applying $(i)$ for shifted rectangles total of $N$ times to obtain
\[
w(R^-_{0,i}(\gamma)) \leq \beta^{-1} w(R^+_{0,i}(\gamma)) \leq \beta^{-N} w(R^-_{N,i}(\gamma)),
\]
where
\[
\beta^{-N} \leq \beta^{-1- 2^{pm} \theta / (1+\gamma) } \leq \beta^{-1 - 2^{p+1} \theta /(1-\gamma)} = C .
\]
Therefore, we conclude that
\begin{align*}
w(R^-(\gamma)) &\leq \sum_i w(R^-_{0,i}(\gamma)) \leq C \sum_i w(R^-_{N,i}(\gamma))\\
&\leq 4 C w(R^-(\gamma) + (0, \theta L^p) )
\end{align*}
by $R^-_{N,i}(\gamma)\subset R^-(\gamma) + (0, \theta L^p )$ and the bounded overlap of $R^-_{N,i}(\gamma)$.
Since $C$ is an increasing function with respect to $\theta $, the claim follows.
\end{proof}

\begin{lemma}
\label{lemma:measureweight-estimate}
Let $0<\gamma<1$ and $w>0$ be a weight.
Assume that there exist $0<\alpha,\beta<1$ 
such that for every parabolic rectangle $R$ and every measurable set $E \subset R^+(\gamma)$ for which $w(E) < \beta w(R^-(\gamma))$ it holds that $\lvert E \rvert < \alpha \lvert R^+(\gamma) \rvert$.
Then there exist $\tau \geq 1$ and $c=c(n,p,\gamma,\alpha,\beta)$ such that for every parabolic rectangle $R=R(x,t,L) \subset \mathbb{R}^{n+1}$ and $\lambda \geq (w_{U^-})^{-1}$ we have
\[
\lvert R^{+}(\gamma) \cap \{ w^{-1} > \lambda \} \rvert \leq c \lambda w( R^{\tau} \cap \{ w^{-1} > (1-\alpha) \lambda \}  ) ,
\]
where $U^- = R^+(\gamma) - (0, \tau (1+\gamma) L^p)$ and
$R^{\tau} = Q(x,L) \times (t+\gamma L^p - \tau(1+\gamma)L^p , t+L^p)$.
Note that $U^- = R^-(\gamma)$ and $R^\tau = R$ for $\tau=1$.
\end{lemma}

\begin{proof}
Let $R_0=R(x_0,t_0,L) = Q(x_0,L) \times (t_0-L^p, t_0+L^p)$.
Denote $f=w^{-1}$ and $\dmu = w \dx \dt$.
Let $\tau\geq 1$ to be chosen later.

Denote $S^+_0 = R^+_0(\gamma)$. The time length of $S^+_0$ is $l_t(S^+_0) = (1-\gamma) L^p$.
We partition $S^+_0$ by dividing each spatial edge into $2$ equally long intervals. If
\[
\frac{l_t(S_{0}^+)}{\lceil 2^{p} \rceil} > \frac{(1-\gamma) L^p}{2^{p}},
\]
we divide the time interval of $S^+_0$ into $\lceil 2^{p} \rceil$ equally long intervals. Otherwise, we divide the time interval of $S^+_0$ into $\lfloor 2^{p} \rfloor$ equally long intervals.
We obtain subrectangles $S^+_1$ of $S^+_0$ with spatial side length $L_1 = l_x(S^+_1) = l_x(S^+_0)/2 = L / 2$
and time length either
\[
l_t(S^+_1) = \frac{l_t(S^+_0)}{\lceil 2^{p} \rceil} = \frac{(1-\gamma) L^p}{\lceil 2^{p} \rceil} \quad \text{or} \quad l_t(S^+_1) = \frac{(1-\gamma) L^p}{\lfloor 2^{p} \rfloor} .
\]
For every $S^+_1$, there exists a unique rectangle $R_1$ with spatial side length $L_1 = L / 2$ 
and time length $2L_1^p = 2 L^p / 2^{p}$
such that $R_1$ has the same top as $S^+_1$.
We select those rectangles $S^+_1$ for which
\[
\frac{\lvert U^-_1 \rvert}{w(U^-_1)} = \dashint_{U^-_1} f \dmu > \lambda
\]
and denote the obtained collection by $\{ S^+_{1,j} \}_j$.
If
\[
\frac{\lvert U^-_1 \rvert}{w(U^-_1)} = \dashint_{U^-_1} f \dmu \leq \lambda ,
\]
we subdivide $S^+_1$ in the same manner as above
and select all those subrectangles $S^+_2$ for which
\[
\frac{\lvert U^-_2 \rvert}{w(U^-_2)} = \dashint_{U^-_2} f \dmu > \lambda
\]
to obtain family $\{ S^+_{2,j} \}_j$.
We continue this selection process recursively.
At the $i$th step, we partition unselected rectangles $S^+_{i-1}$ by dividing each spatial side into $2$ equally long intervals. If
\begin{equation}
\label{eq:qualiproof_eq1}
\frac{l_t(S_{i-1}^+)}{\lceil 2^{p} \rceil} > \frac{(1-\gamma) L^p}{2^{pi}},
\end{equation}
we divide the time interval of $S^+_{i-1}$ into $\lceil 2^{p} \rceil$ equally long intervals. 
Otherwise, if
\begin{equation}
\label{eq:qualiproof_eq2}
\frac{l_t(S_{i-1}^+)}{\lceil 2^{p} \rceil} \leq \frac{(1-\gamma) L^p}{2^{pi}},
\end{equation}
we divide the time interval of $S^+_{i-1}$ into $\lfloor 2^{p} \rfloor$ equally long intervals.
We obtain subrectangles $S^+_i$. For every $S^+_i$, there exists a unique rectangle $R_i$ with spatial side length $L_i = L / 2^{i}$
and time length $2L_i^p = 2 L^p / 2^{pi}$
such that $R_i$ has the same top as $S^+_i$.
Select those $S^+_i$ for which 
\[
\frac{\lvert U^-_i \rvert}{w(U^-_i)} = \dashint_{U^-_i} f \dmu > \lambda
\]
and denote the obtained collection by $\{ S^+_{i,j} \}_j$.
If
\[
\frac{\lvert U^-_i \rvert}{w(U^-_i)} = \dashint_{U^-_i} f \dmu \leq \lambda ,
\]
we continue the selection process in $S^+_i$.
In this manner we obtain a collection $\{S^+_{i,j} \}_{i,j}$ of pairwise disjoint rectangles.

Observe that if \eqref{eq:qualiproof_eq1} holds, then we have
\[
l_t(S_i^+) = \frac{l_t(S^+_{i-1})}{\lceil 2^{p} \rceil} \geq \frac{(1-\gamma) L^p}{2^{pi}}.
\]
On the other hand, if \eqref{eq:qualiproof_eq2} holds, then
\[
l_t(S_i^+) = \frac{l_t(S^+_{i-1})}{\lfloor 2^{p} \rfloor} \geq \frac{l_t(S^+_{i-1})}{2^{p}} \geq \dots \geq \frac{(1-\gamma)  L^p}{2^{pi}} .
\]
This gives a lower bound 
\[
l_t(S_i^+) \geq \frac{ (1-\gamma) L^p}{2^{pi}}
\]
for every $S_i^+$.

Suppose that \eqref{eq:qualiproof_eq2} is satisfied at the $i$th step.
Then we have an upper bound for the time length of $S_i^+$, since
\begin{align*}
l_t(S^+_i) = \frac{l_t(S_{i-1}^+)}{\lfloor 2^{p} \rfloor} \leq \frac{\lceil 2^{p} \rceil}{\lfloor 2^{p} \rfloor} \frac{(1-\gamma) L^p}{2^{pi}} \leq \biggl( 1+ \frac{1}{\lfloor 2^{p} \rfloor} \biggr) \frac{(1-\gamma) L^p}{2^{pi}} .
\end{align*}
On the other hand, if \eqref{eq:qualiproof_eq1} is satisfied, then
\[
l_t(S^+_i) = \frac{l_t(S_{i-1}^+)}{\lceil 2^{p} \rceil} \leq \frac{l_t(S_{i-1}^+)}{2^{p}}.
\]
In this case, \eqref{eq:qualiproof_eq2} has been satisfied at an earlier step $i'$ with $i'< i$.
We obtain
\begin{align*}
l_t(S^+_i) \leq \frac{l_t(S_{i-1}^+)}{ 2^{p}} \leq \dots \leq \frac{l_t(S_{i'}^+)}{ 2^{p(i-i')}} \leq \biggl( 1+ \frac{1}{\lfloor 2^{p} \rfloor} \biggr) \frac{(1-\gamma) L^p}{ 2^{pi}}
\end{align*}
by using the upper bound for $S_{i'}^+$.
Thus, we have
\[
\frac{(1-\gamma) L^p}{2^{pi}} \leq l_t(S^+_i) \leq \biggl( 1+ \frac{1}{\lfloor 2^{p} \rfloor} \biggr) \frac{(1-\gamma) L^p}{2^{pi}}
\]
for every $S^+_i$.
For each $S^+_i$, let $0 \leq \varepsilon_i \leq 1 /\lfloor 2^{p} \rfloor$ such that
\[
l_t(S^+_i) = (1+ \varepsilon_i) l_t(R^+_i(\gamma)) = (1+ \varepsilon_i) \frac{(1-\gamma) L^p}{2^{pi}} .
\]
Note that $\varepsilon_0 = 0$.

We have a collection $\{ S^+_{i,j} \}_{i,j}$ of pairwise disjoint rectangles. 
However, the rectangles in the corresponding collection $\{ U^-_{i,j} \}_{i,j}$ may overlap. 
Thus, we replace it by a subfamily $\{ \widetilde{U}^-_{i,j} \}_{i,j}$ of pairwise disjoint rectangles, which is constructed in the following way.
At the first step, choose $\{ U^-_{1,j} \}_{j}$ and denote it by $\{ \widetilde{U}^-_{1,j} \}_j$. 
Then consider the collection $\{ U^-_{2,j} \}_{j}$ where each $U^-_{2,j}$ either intersects some $\widetilde{U}^-_{1,j}$ or does not intersect any $\widetilde{U}^-_{1,j}$. 
Select the rectangles $U^-_{2,j}$ that do not intersect any $\widetilde{U}^-_{1,j}$, and denote the obtained collection by $\{ \widetilde{U}^-_{2,j} \}_j$.
At the $i$th step, choose those $U^-_{i,j}$ that do not intersect any previously selected $\widetilde{U}^-_{i',j}$, $i' < i$.
Hence, we obtain a collection $\{ \widetilde{U}^-_{i,j} \}_{i,j}$ of pairwise disjoint rectangles.
Observe that for every $U^-_{i,j}$ there exists $\widetilde{U}^-_{i',j}$ with $i' < i$ such that
\begin{equation}
\label{eq:qualiproof_plussubset}
\text{pr}_x(U^-_{i,j}) \subset \text{pr}_x(\widetilde{U}^-_{i',j}) \quad \text{and} \quad \text{pr}_t(U^-_{i,j}) \subset 3 \text{pr}_t(\widetilde{U}^-_{i',j}) .
\end{equation}
Here pr$_x$ denotes the projection to $\mathbb R^n$ and pr$_t$ denotes the projection to the time axis.
Note that $S^+_{i,j}$ is spatially contained in $U^-_{i,j}$, that is, $\text{pr}_x S^+_{i,j}\subset \text{pr}_x U^-_{i,j}$.
In the time direction, we have
\begin{equation}
\label{eq:qualiproof_minusplussubset}
\text{pr}_t(S^+_{i,j}) \subset \text{pr}_t(R^\tau_{i,j}) 
\subset \biggl( 2\tau \frac{1 + \gamma}{1-\gamma} +1 \biggr) \text{pr}_t(U^-_{i,j}) ,
\end{equation}
since
\[
\biggl( 2\tau \frac{1 + \gamma}{1-\gamma} + 2 \biggr) \frac{l_t(U^-_{i,j})}{2} = \frac{(1-\gamma)L^p}{2^{pi}} + \frac{\tau(1+\gamma)L^p}{2^{pi}} = l_t(R^{\tau}_{i,j}) ,
\]
where recall that $R^{\tau}_{i,j} = Q(x_{R_{i,j}} ,L_i) \times (t_{R_{i,j}} +\gamma L_i^p - \tau(1+\gamma)L_i^p , t_{R_{i,j}} +L_i^p)$.
Therefore, by~\eqref{eq:qualiproof_plussubset} and~\eqref{eq:qualiproof_minusplussubset}, it holds that
\begin{equation}
\label{eq:qualiproof_start}
\Big\lvert \bigcup_{i,j} S^+_{i,j} \Big\rvert \leq c_1 \sum_{i,j} \lvert \widetilde{U}^-_{i,j} \rvert 
\quad\text{with}\quad
c_1 = 3 \biggl( 2\tau \frac{1 + \gamma}{1-\gamma} +1 \biggr).
\end{equation}

For the rest of the proof and to simplify the notation, let $U^-_i = \widetilde{U}^-_{i,j}$ and $U^-_{i-1} = \widetilde{U}^-_{i-1,j'}$ be fixed, where $U^-_i$ was obtained by subdividing the previous $U^-_{i-1}$ for which $\lvert U^-_{i-1} \rvert / w(U^-_{i-1}) \leq \lambda$. 
Choose $N\in\mathbb N$ and $\eta>1$ such that
\[
\alpha^N \leq 2^{-n-p} < \alpha^{N-1} \quad\text{and}\quad \eta^{n+p} = \alpha^{-1} . 
\]
Then we have
\begin{align*}
\lvert U^-_{i} \rvert = \frac{1}{2^{n+p}} \lvert U^-_{i-1} \rvert \geq \alpha^N \lvert U^-_{i-1} \rvert \quad\text{and}\quad \eta^{N-1} < 2 \leq \eta^N .
\end{align*}
We construct a chain of rectangles from $U^-_i$ to $U^-_{i-1}$.
Define the first element of the chain by
\[
V_0 = 
U^-_i - (1+\gamma) L_i^p 
= Q(x_{R_{i}},L_i) \times (a_0 , a_0 + (1-\gamma) L_i^p) ,
\]
where $a_0$ denotes the time coordinate of the bottom of $V_0$.
For the rest of the chain, we consider separately the spatial variable and the time variable.
We start with the spatial variable.
Let
\[
Q_k = \eta^k Q(x_{R_{i}},L_i) + \frac{\eta^{k}-1}{\eta^{N}-1}(x_{R_{i-1}} - x_{R_{i}})
\]
for $k\in\{0,\dots,N\}$.
We have
\[
\lvert x_{R_{i-1},m} - x_{R_{i},m} \rvert \leq \eta^N L_i - L_i = (\eta^N-1) L_i
\]
for every $m\in \{1,\dots,n\}$, since $pr_x(R_{i}) \subset pr_x(R_{i-1})$.
Here $x_{R_{i},m}$ denotes the $m$th coordinate of $x_{R_{i}}$.
Then $Q_{k-1} \subset Q_{k}$ for every $k\in\{1,\dots,N\}$ since
\begin{align*}
&\lvert x_{Q_{k-1},m} - x_{Q_{k},m} \rvert \\
&\qquad= \Bigl\lvert x_{R_{i},m} + \frac{\eta^{k-1}-1}{\eta^{N}-1}(x_{R_{i-1},m} - x_{R_{i},m})  - x_{R_{i},m} - \frac{\eta^{k}-1}{\eta^{N}-1}(x_{R_{i-1},m} - x_{R_{i},m}) \Bigr\rvert \\
&\qquad= \frac{\eta^{k}-\eta^{k-1}}{\eta^N -1} \lvert x_{R_{i-1},m} - x_{R_{i},m} \rvert 
\leq (\eta^{k}-\eta^{k-1}) L_i = l(Q_{k}) - l(Q_{k-1})
\end{align*}
for every $m\in \{1,\dots,n\}$.
Observe that $Q_0 = pr_x(U^-_{i})$ and $Q_N = pr_x(U^-_{i-1})$.

We move on to the time variable where the chain is constructed as follows.
Let
\[
I_k = (a_k,b_k) = (a_{k-1} - \eta^{pk} (1+\gamma) L_i^p ,  a_{k-1} + \eta^{pk} (1-\gamma) L_i^p - \eta^{pk} (1+\gamma) L_i^p  )
\]
for $k\in \{1,\dots,N\}$.
We define the elements of the chain by $V_k = Q_{k} \times I_k$
for every $k\in \{1,\dots,N\}$. Observe that $\lvert V_{k-1} \rvert = \alpha \lvert V_{k} \vert$.
The time distance between the bottom of $U^-_i$ and the bottom of $V_N$ is
\begin{align*}
\sum_{k=0}^{N} \eta^{pk} (1+\gamma) L_i^p = \frac{\eta^{p(N+1)}-1}{\eta^p -1} (1+\gamma) L_i^p = \frac{\eta^{p(N+1)}-1}{\eta^p -1} \frac{(1+\gamma) L^p}{2^{pi}} .
\end{align*}
Let $\tau\geq 1$ such that
\begin{align*}
\tau (1+\gamma) L^p = \frac{\tau (1+\gamma) L^p}{2^p} + \frac{\eta^{p(N+1)}-1}{\eta^p -1} \frac{(1+\gamma) L^p}{2^{p}}
+ \frac{(1-\gamma) L^p}{\lfloor 2^p \rfloor}  ,
\end{align*}
that is,
\begin{equation}
\label{eq:qualiproof_tau}
\tau = \frac{2^p}{2^p-1} \frac{\eta^{p(N+1)}-1}{\eta^p -1}
+ \frac{2^p}{2^p-1} \frac{1}{\lfloor 2^p \rfloor} \frac{1-\gamma}{1+\gamma}   .
\end{equation}
With this choice, we must add one more rectangle into the chain for the chain to end at $U^-_{i-1}$.
Let
\begin{align*}
V_{N+1} &= V_N - (0, b_i l_t(V_N) ) = V_N - (0, b_i \eta^{Np} (1-\gamma) L_{i}^p ) \\
&= V_N - \biggl( 0, b_i \frac{ \eta^{Np}(1-\gamma) L^p}{2^{pi}} \biggr) , 
\end{align*}
where $b_i$ is chosen
such that the bottom of $V_{N+1}$ coincides with the bottom of $U^-_{i-1}$. Then $V_{N+1}$ contains $U^-_{i-1}$.
It must hold that
\begin{equation}
\label{eq:qualiproof_rectangle-equation}
l_t(R^\tau_{i-1}) - l_t(S^+_{i-1}) = l_t(R^\tau_i) - l_t(S^+_i) + d_i - K_i l_t(S^+_i) .
\end{equation}
Here $K_i$ measures the time length between the bottom of $S^+_{i}$ and the bottom of $S^+_{i-1}$, and $K_i \in \{0,\dots, \lceil 2^p \rceil -1\}$ or $K_i \in \{0,\dots, \lfloor 2^p \rfloor -1\}$ depending on the partition level $i$.
If the bottom of $S^+_{i}$ intersects the bottom of $S^+_{i-1}$, then $K_i=0$.
Moreover, $d_i$ is time distance from the bottom of $U^-_i$ to the bottom of $U^-_{i-1}$ or equivalently $V_{N+1}$, that is,
\[
d_i = \frac{\eta^{p(N+1)}-1}{\eta^p -1} \frac{(1+\gamma) L^p}{2^{pi}} + b_i \frac{\eta^{Np} (1-\gamma) L^p}{2^{pi}} .
\]
Equation~\eqref{eq:qualiproof_rectangle-equation} is equivalent with
\begin{align*}
&\frac{\tau (1+\gamma) L^p}{2^{p(i-1)}} - \varepsilon_{i-1} \frac{(1-\gamma) L^p}{2^{p(i-1)}} = \frac{\tau (1+\gamma) L^p}{2^{pi}} - \varepsilon_i \frac{(1-\gamma) L^p}{2^{pi}} \\
&\qquad + \frac{\eta^{p(N+1)}-1}{\eta^p -1} \frac{(1+\gamma) L^p}{2^{pi}} + b_i \frac{(1-\gamma) L^p}{2^{p(i-1)}}   - K_i (1+ \varepsilon_i) \frac{(1-\gamma) L^p}{2^{pi}} .
\end{align*}
By the choice of $\tau$, 
we deduce that
\begin{align*}
&b_i \frac{ \eta^{Np}(1-\gamma) L^p}{2^{pi}}   = 
\frac{(1-\gamma) L^p}{2^{p(i-1)} \lfloor 2^p \rfloor}   + \varepsilon_i \frac{(1-\gamma) L^p}{2^{pi}} \\
&\qquad - \varepsilon_{i-1} \frac{(1-\gamma) L^p}{2^{p(i-1)}}  + K_i (1+ \varepsilon_i) \frac{(1-\gamma) L^p}{2^{pi}} ,
\end{align*}
from which we obtain
\begin{align*}
0 &\leq b_i = \frac{1}{\eta^{Np}} \biggl( \frac{2^p}{\lfloor 2^p \rfloor}  + \varepsilon_i - 2^p \varepsilon_{i-1} + K_i (1+ \varepsilon_j) \biggr) \\
&\leq \frac{1}{2^p} \biggl( \frac{2^p}{\lfloor 2^p \rfloor} + \frac{1}{\lfloor 2^p \rfloor} + ( \lceil 2^p \rceil -1 ) \biggl( 1+ \frac{1}{\lfloor 2^p \rfloor} \biggr) \biggr)
\leq 3 .
\end{align*}
With this choice of $b_i$, we have $U^-_{i-1} \subset V_{N+1}$.

It holds that
\begin{align*}
\lvert U^-_{i} \cap \{ (1-\alpha) w \geq w_{U^-_{i}} \} \rvert \leq \frac{1-\alpha}{w_{U^-_{i}}} w(U^-_{i}) = (1-\alpha) \lvert U^-_{i} \rvert ,
\end{align*}
and thus
\begin{align*}
\lvert U^-_{i} \cap \{ (1-\alpha) w < w_{U^-_{i}} \} \rvert \geq \alpha \lvert U^-_{i} \rvert = \alpha \lvert V_0 \rvert .
\end{align*}
Hence, we may apply Lemma~\ref{lemma:doublingforward}~$(i)$ to get
\[
w(U^-_{i} \cap \{ (1-\alpha) w < w_{U^-_{i}} \}) \geq \beta w(V_0) .
\]
By the definition of $V_k$, we can apply Lemma~\ref{lemma:doublingforward}~$(i)$ for each pair of $V_{k-1}, V_k$, $k\in\{1,\dots,N\}$, and Lemma~\ref{lemma:doublingforward}~$(ii)$ for $V_{N}, V_{N+1}$ with $\theta=3\geq\sup b_i$ to obtain
\[
w(V_0) \geq \beta w(V_1) \geq \beta^N w(V_N) \geq \beta^N \frac{1}{C} w(V_{N+1}) .
\]
By combining the previous estimates, $U^-_{i-1} \subset V_{N+1}$ and $w_{U^-_{i}} < \lambda^{-1} $, we conclude that
\begin{equation}
\label{eq:qualiproof_chainestimate}
\begin{split}
w(U^-_{i-1}) &\leq w(V_{N+1}) \leq C \beta^{-N-1} w(U^-_{i} \cap \{ (1-\alpha) w < w_{U^-_{i}} \}) \\
&\leq c_2 w(U^-_{i} \cap \{ w^{-1} > (1-\alpha) \lambda \}) ,
\end{split}
\end{equation}
where
\[
C \beta^{-N-1} \leq 4 \beta^{-1 - 2^{p+1} 3 /(1-\gamma) -2+(n+p)\log_\alpha 2 } = c_2 .
\]

If $(x,t) \in S^+_0 \setminus \bigcup_{i,j} S^+_{i,j}$, then there exists a sequence of subrectangles $S^+_l$ containing $(x,t)$ such that 
\[
\frac{\lvert U^-_l \rvert}{w(U^-_l)} = \dashint_{U^-_l} f \dmu \leq \lambda
\]
and $\lvert S^+_l \rvert \to 0$ as $l \to \infty$.
The Lebesgue differentiation theorem~\cite[Lemma~2.3]{KinnunenMyyryYang2022}
implies that $w^{-1} = f(x,t) \leq \lambda$
for almost every $(x,t) \in S^+_0 \setminus \bigcup_{i,j} S^+_{i,j}$.
It follows that
\[
S^+_0 \cap \{ w^{-1} > \lambda \} \subset \bigcup_{i,j} S^+_{i,j}
\]
up to a set of measure zero.
Using this together with~\eqref{eq:qualiproof_start} and~\eqref{eq:qualiproof_chainestimate}, we obtain
\begin{align*}
\lvert S^+_0 \cap \{ w^{-1} > \lambda \} \rvert &\leq 
c_1 \sum_{i,j} \lvert \widetilde{U}^-_{i,j} \rvert = \frac{c_1}{2^{n+p} } \sum_{i,j} \lvert \widetilde{U}^-_{i-1,j} \rvert \leq \frac{c_1}{2^{n+p}} \lambda \sum_{i,j} w(\widetilde{U}^-_{i-1,j}) \\
&\leq \frac{c_1}{2^{n+p}} c_2 \lambda \sum_{i,j} w(\widetilde{U}^-_{i,j} \cap \{ w^{-1} > (1-\alpha) \lambda \}) \\
&\leq c \lambda w(R_0^{\tau} \cap \{ w^{-1} > (1-\alpha) \lambda \}) ,
\end{align*}
where $c = c_1 c_2 /2^{n+p}$.
This completes the proof.
\end{proof}

The following theorem shows that the qualitative measure condition in Theorem~\ref{thm:Ainfty}~$(iv)$ implies the parabolic Muckenhoupt condition in Theorem~\ref{thm:Ainfty}~$(i)$.

\begin{theorem}
Let $0<\gamma<1$ and $w$ be a weight in $\mathbb{R}^{n+1}$. 
Assume that there exist $0<\alpha, \beta<1$ such that for every parabolic rectangle $R$ and every measurable set $E \subset R^+(\gamma)$ for which $w(E) < \beta w(R^-(\gamma))$ it holds that $\lvert E \rvert < \alpha \lvert R^+(\gamma) \rvert$.
Then $w\in A^+_q(\gamma)$ for some $q>1$.
\end{theorem}

\begin{proof}
We prove the claim first for weights satisfying that $w^{-1}$ is bounded. 
Let $R \subset \mathbb R^{n+1}$ be a parabolic rectangle.
Let $\varepsilon>0$ to be chosen later.
Denote $B = (w_{U^-})^{-1}$ and $\rho = 1-\alpha$.
Applying Cavalieri's principle with Lemma~\ref{lemma:measureweight-estimate}, we obtain
\begin{align*}
\int_{R^+(\gamma)} w^{-\varepsilon} &= \varepsilon \int_0^\infty \lambda^{\varepsilon-1} \lvert R^{+}(\gamma) \cap \{ w^{-1} > \lambda \} \rvert \dla \\
&= \varepsilon \int_0^{B} \lambda^{\varepsilon-1} \lvert R^{+}(\gamma) \cap \{ w^{-1} > \lambda \} \rvert \dla + \varepsilon \int_{B}^\infty \lambda^{\varepsilon-1} \lvert R^{+}(\gamma) \cap \{ w^{-1} > \lambda \} \rvert \dla \\
&\leq \lvert R^{+}(\gamma) \rvert \varepsilon \int_0^{B} \lambda^{\varepsilon-1} \dla + c \varepsilon \int_{B}^\infty \lambda^{\varepsilon} w( R^{\tau} \cap \{ w^{-1} > \rho \lambda \}  ) \dla \\
&\leq \lvert R^{+}(\gamma) \rvert B^\varepsilon + c \rho^{-1-\varepsilon} \varepsilon \int_0^\infty \lambda^{\varepsilon} w( R^{\tau} \cap \{ w^{-1} > \lambda \}  ) \dla \\
&\leq \lvert U^- \rvert (w_{U^-})^{-\varepsilon} + c \rho^{-1-\varepsilon} \frac{\varepsilon}{1+\varepsilon} \int_{R^{\tau}} w^{-\varepsilon} ,
\end{align*}
where $c$ is the constant from Lemma~\ref{lemma:measureweight-estimate}.
By choosing $\varepsilon>0$ to be small enough, we can absorb the integral over $R^+(\gamma)$ of the second term to the left-hand side to get
\begin{align*}
\biggl( 1- \frac{c}{\rho^{1+\varepsilon}} \frac{\varepsilon}{1+\varepsilon} \biggr) \int_{R^+(\gamma)} w^{-\varepsilon} \leq \lvert U^- \rvert (w_{U^-})^{-\varepsilon} + \frac{c}{\rho^{1+\varepsilon}} \frac{\varepsilon}{1+\varepsilon} \int_{R^{\tau}\setminus R^+(\gamma)} w^{-\varepsilon} .
\end{align*}
Denote $R^{\tau,-} = R^{\tau} \setminus R^+(\gamma)$.
Hence, we have
\begin{equation}
\label{eq:qualiproof_cavalieri_iteration}
\int_{R^+(\gamma)} w^{-\varepsilon} \leq c_0 \lvert U^- \rvert (w_{U^-})^{-\varepsilon} + c_1 \varepsilon \int_{R^{\tau,-}} w^{-\varepsilon} ,
\end{equation}
where
\[
c_0 = \frac{1+\varepsilon }{1-(c \rho^{-1-\varepsilon} -1) \varepsilon} \quad \text{and} \quad c_1 = \frac{c \rho^{-1-\varepsilon} }{1-(c \rho^{-1-\varepsilon} -1) \varepsilon} .
\]

Fix $R_0 = Q(x_0,L) \times (t_0 - L^p, t_0 + L^p) \subset \mathbb R^{n+1}$.
We cover $R^{\tau,-}_{0}(\gamma)$ by 
\[
M = 2^n \biggl\lceil \frac{\tau(1+\gamma)}{(1-\gamma)/2^p} \biggr\rceil = 2^{n} \biggl\lceil 2^p \tau \frac{1+\gamma}{1-\gamma} \biggr\rceil
\]
rectangles $R^+_{1,j}(\gamma)$ with spatial side length $ L_1 = L/2$ and time length $ (1-\gamma)L_1^p = (1-\gamma) L^p / 2^p$. This can be done by dividing each spatial edge of $R^{\tau,-}_{0}(\gamma)$ into two equally long pairwise disjoint intervals, and the time interval of $R^{\tau,-}_{0}(\gamma)$ into $\lceil 2^p \tau (1+\gamma) /(1-\gamma) \rceil$ equally long intervals such that their overlap is bounded by $2$.
Thus, the overlap of $R^+_{1,j}(\gamma)$ is bounded by $2$.
Then consider $R^{\tau,-}_{1,j}(\gamma)$ and cover it in the same way as before by $M$ rectangles $R^+_{2,j}(\gamma)$ with spatial side length $ L_2 = L/2^{2}$ and time length $ (1-\gamma)L_2^p = (1-\gamma)L^p / 2^{2p}$.
At the $i$th step, cover $R^{\tau,-}_{i-1,j}(\gamma)$ by $M$ rectangles $R^+_{i,j}(\gamma)$ with spatial side length $ L_i = L/2^{i}$ and time length $ (1-\gamma)L_i^p = (1-\gamma)L^p / 2^{pi}$ such that their overlap is bounded by $2$ for fixed $R^{\tau,-}_{i-1,j}(\gamma)$.
We observe that the time distance between the bottom of $U^-_{i,j}$ and the bottom of $R_0^+(\gamma)$ is at most
\begin{equation}
\label{eq:qualiproof_maxdist-lowerparts}
\sum_{i=0}^\infty l_t(R^{\tau,-}_{i,j}(\gamma)) = \sum_{i=0}^\infty \frac{\tau(1+\gamma)L^p}{2^{pi}} = \frac{2^p}{2^p -1} \tau(1+\gamma)L^p .
\end{equation}

As in the proof of Lemma~\ref{lemma:measureweight-estimate}, we construct a chain from each $U^-_{i,j}$ to 
$U^{\sigma,-}_{0} = R^+(\gamma) - (0, \sigma (1+\gamma) L^p )$, where $\sigma\geq\tau$ is chosen later.
Fix $U^-_{i} = U^-_{i,j}$.
Choose $N = N(i) \in\mathbb N$ and $\eta>1$ such that
\[
\alpha^N \leq 2^{-(n+p)i} < \alpha^{N-1} \quad\text{and}\quad \eta^{n+p} = \alpha^{-1} . 
\]
Then we have
\begin{align*}
\lvert U^-_{i} \rvert = \frac{1}{2^{(n+p)i}} \lvert U^{\sigma,-}_{0} \rvert \geq \alpha^{N} \lvert U^{\sigma,-}_{0} \rvert \quad\text{and}\quad \eta^{N-1} < 2^i \leq \eta^N .
\end{align*}
Let $0\leq\delta=\delta(i)<1$ such that $\eta^{N-\delta} = 2^i $.
We construct a chain of rectangles from $U^-_i$ to $U^{\sigma,-}_{0}$.
Define its elements by
\begin{align*}
V_0 = U^-_i = Q(x_{R_{i}},L_i) \times (a_0 , a_0 + (1-\gamma) L_i^p) \quad\text{and}\quad V_k = Q_{k} \times I_k
\end{align*}
for every $k\in \{1,\dots,N\}$, where
\begin{align*}
Q_k &= \eta^k Q(x_{R_{i}},L_i) + \frac{\eta^{k}-1}{\eta^{N}-1}(x_{R_{0}} - x_{R_{i}}) , \\
I_k &= (a_k,b_k) = (a_{k-1} - \eta^{kp} (1+\gamma) L_i^p ,  a_{k-1} + \eta^{kp} (1-\gamma) L_i^p - \eta^{kp} (1+\gamma) L_i^p  ) .
\end{align*}
Observe that $Q_0 = pr_x(U^-_{i})$, $Q_{N} = pr_x(U^-_{0,\sigma})$ and $\lvert V_{k-1} \rvert = \alpha \lvert V_{k} \vert$.
The time distance between the bottom of $V_0$ and the bottom of $V_N$ is
\begin{align*}
\sum_{k=1}^{N} \eta^{pk} (1+\gamma) L_i^p &= \eta^p \frac{\eta^{pN}-1}{\eta^p -1} \frac{(1+\gamma) L^p}{2^{pi}} = \frac{\eta^{p}}{\eta^p -1} \frac{ 2^{pi} \eta^{\delta p} -1}{2^{pi}} (1+\gamma) L^p \\
&= \frac{\eta^{p}}{\eta^p -1} \biggl( \eta^{\delta p} - \frac{1}{2^{pi}} \biggr) (1+\gamma) L^p .
\end{align*}
Hence, the maximum possible distance between the bottom of $V_0$ and the bottom of $V_N$ is
\begin{equation}
\label{eq:qualiproof_maxdist-chain}
\sum_{k=1}^{\infty} \eta^{pk} (1+\gamma) L_i^p = \frac{\eta^{p+\delta p}}{\eta^p -1} (1+\gamma) L^p \leq \frac{\eta^{2p}}{\eta^p -1} (1+\gamma) L^p .
\end{equation}
By combining~\eqref{eq:qualiproof_maxdist-lowerparts} and~\eqref{eq:qualiproof_maxdist-chain}, 
we obtain an upper bound for the time length
from the bottom of $R_0^+(\gamma)$ to the bottom of $V_N$. Based on this, we fix $U^{\sigma,-}_{0}$ by choosing $\sigma$ such that
\[
\sigma (1+\gamma)L^p = \frac{2^p}{2^p -1} \tau(1+\gamma)L^p +  \frac{\eta^{2p}}{\eta^p -1} (1+\gamma) L^p ,
\]
that is,
\[
\sigma = \frac{2^p\tau}{2^p -1} + \frac{\eta^{2p}}{\eta^p -1} .
\]

We add one more rectangle $V_{N+1}$ into the chain so that the chain would end at $U^{\sigma,-}_{0}$.
Let
\begin{align*}
V_{N+1} &= V_N - (0, b_i l_t(V_N) ) = V_N - (0, b_i \eta^{Np} (1-\gamma) L_{i}^p )\\
&=V_N - \biggl(0, b_i \frac{ \eta^{Np}(1-\gamma) L^p}{2^{pi}} \biggr) , 
\end{align*}
where $b_i$ is chosen such that the bottom of $V_{N+1}$ intersects with the bottom of $U^{\sigma,-}_{0}$.
Then $U^{\sigma,-}_{0}$ is contained in $V_{N+1}$.
Next we find an upper bound for $b_i$.
By recalling the definition of $\tau$~\eqref{eq:qualiproof_tau} from the proof of Lemma~\ref{lemma:measureweight-estimate}, 
we observe that
the shortest possible time length from the bottom of $R^+_0(\gamma)$ to the bottom $V_N$ is
\begin{align*}
\sum_{i=1}^\infty \frac{(1-\gamma)L^p}{2^{pi}} + \sum_{k=1}^{\infty} \eta^{pk} (1+\gamma) L_i^p &= \frac{1-\gamma}{2^p -1} L^p + \frac{\eta^{p+\delta p}}{\eta^p -1} (1+\gamma) L^p \\
&\geq \frac{1-\gamma}{2^p -1} L^p + \frac{\eta^{p}}{\eta^p -1} (1+\gamma) L^p .
\end{align*}
Therefore, the time distance between the bottom of $V_N$ and the bottom of $U^{\sigma,-}_{0}$ is less than
\begin{align*}
&\sigma (1+\gamma)L^p - \frac{1-\gamma}{2^p -1} L^p - \frac{\eta^{p}}{\eta^p -1} (1+\gamma) L^p \\
&\qquad= \frac{2^p}{2^p -1} \tau(1+\gamma)L^p - \frac{1-\gamma}{2^p -1} L^p + \eta^p (1+\gamma) L^p .
\end{align*}
By this,
we obtain an upper bound for $b_i$
\begin{align*}
b_i (1-\gamma) L^p 
&\leq b_i \frac{ \eta^{Np}(1-\gamma) L^p}{2^{pi}}\\
&\leq \frac{2^p}{2^p -1} \tau(1+\gamma)L^p - \frac{1-\gamma}{2^p -1} L^p + \eta^p (1+\gamma) L^p ,
\end{align*}
that is,
\[
b_i \leq \frac{2^p \tau }{2^p -1} \frac{1+\gamma}{1-\gamma} - \frac{1}{2^p -1} + \eta^p \frac{1+\gamma}{1-\gamma} = \theta .
\]

By the definition of $V_k$, we can apply Lemma~\ref{lemma:doublingforward}~$(i)$ for each pair of $V_{k-1}, V_k$, $k\in\{1,\dots,N\}$, and Lemma~\ref{lemma:doublingforward}~$(ii)$ for $V_{N}, V_{N+1}$ with $\theta\geq\sup b_i$ to get
\[
w(V_0) \geq \beta w(V_1) \geq \beta^N w(V_N) \geq \beta^N \frac{1}{C} w(V_{N+1}) ,
\]
where $C$ is the constant from Lemma~\ref{lemma:doublingforward}~$(ii)$.
Note that $\beta^{-N} \leq \beta^{-1+(n+p)i\log_\alpha 2 }$.
We conclude that
\begin{equation}
\label{eq:qualiproof_chainestimate2}
w(U^{\sigma,-}_{0}) \leq w(V_{N+1}) \leq C \beta^{-N} w(V_0) \leq c_2 \xi^i w(U^-_i) ,
\end{equation}
where $c_2 = C \beta^{-1}$ and $\xi = \beta^{(n+p)\log_\alpha 2 }$.

We iterate~\eqref{eq:qualiproof_cavalieri_iteration} to obtain
\begin{align*}
\int_{R^+_0(\gamma)} w^{-\varepsilon} &\leq c_0 \lvert U^-_0 \rvert (w_{U^-_0})^{-\varepsilon} + c_1 \varepsilon \int_{R^{\tau,-}_{0}} w^{-\varepsilon} \\
&\leq c_0 \lvert U^-_0 \rvert (w_{U^-_0})^{-\varepsilon} + c_1 \varepsilon \sum_{j=1}^M \int_{R^+_{1,j}(\gamma)} w^{-\varepsilon} \\
&\leq c_0 \lvert U^-_0 \rvert (w_{U^-_0})^{-\varepsilon} + c_1 \varepsilon \sum_{j=1}^M \biggl( c_0 \lvert U^-_{1,j} \rvert (w_{U^-_{1,j}})^{-\varepsilon} + c_1 \varepsilon \int_{R^{\tau,-}_{1,j}(\gamma)} w^{-\varepsilon} \biggr) \\
&= c_0 \lvert U^-_0 \rvert (w_{U^-_0})^{-\varepsilon} + c_0  c_1 \varepsilon \sum_{j=1}^M \lvert U^-_{1,j} \rvert (w_{U^-_{1,j}})^{-\varepsilon} + (c_1 \varepsilon)^2 \sum_{j=1}^M \int_{R^{\tau,-}_{1,j}(\gamma)} w^{-\varepsilon} \\
&\leq c_0 \sum_{i=0}^N \biggl( (c_1 \varepsilon)^{i} \sum_{j=1}^{M^i} \lvert U^-_{i,j} \rvert (w_{U^-_{i,j}})^{-\varepsilon} \biggr) + (c_1 \varepsilon)^{N+1} \sum_{j=1}^{M^N} \int_{R^{\tau,-}_{N,j}(\gamma)} w^{-\varepsilon} \\
&\leq c_0 \sum_{i=0}^N \biggl( (c_1 \varepsilon)^{i} \sum_{j=1}^{M^i} \lvert U^-_{i,j} \rvert (w_{U^-_{i,j}})^{-\varepsilon} \biggr) + (c_1 \varepsilon)^{N+1} M^N \int_{R^{\sigma,-}_{0}(\gamma)} w^{-\varepsilon} \\
&= I + II .
\end{align*}
We observe that $II$ tends to zero if $\varepsilon < \tfrac{1}{c_1 M}$ as $N \to \infty$ since $w^{-\varepsilon}$ is bounded by the initial assumption.
For the inner sum of the first term $I$, we apply~\eqref{eq:qualiproof_chainestimate2} to get
\begin{align*}
\sum_{j=1}^{M^i} \lvert U^-_{i,j} \rvert (w_{U^-_{i,j}})^{-\varepsilon} &= \sum_{j=1}^{M^i} \lvert U^-_{i,j} \rvert^{1+\varepsilon} w(U^-_{i,j})^{-\varepsilon} \leq \sum_{j=1}^{M^i} 2^{-(n+p)(1+\varepsilon) i} \lvert U^{\sigma,-}_{0} \rvert^{1+\varepsilon} w(U^-_{i,j})^{-\varepsilon} \\
&\leq \sum_{j=1}^{M^i} 2^{-(n+p)(1+\varepsilon) i} \lvert U^{\sigma,-}_{0} \rvert^{1+\varepsilon} c_2^\varepsilon \xi^{\varepsilon i} w(U^{\sigma,-}_{0})^{-\varepsilon} \\
&= 2^{-(n+p)(1+\varepsilon) i}  c_2^\varepsilon \xi^{\varepsilon i} M^i \lvert U^{\sigma,-}_{0} \rvert (w_{U^{\sigma,-}_{0}})^{-\varepsilon} .
\end{align*}
Thus, it follows that
\begin{align*}
I &\leq c_0 \sum_{i=0}^N (c_1 \varepsilon)^{i} 2^{-(n+p)(1+\varepsilon) i} c_2^\varepsilon \xi^{\varepsilon i} M^i \lvert U^{\sigma,-}_{0} \rvert (w_{U^{\sigma,-}_{0}})^{-\varepsilon} \\
&\leq c_0 c_2^\varepsilon \lvert U^{\sigma,-}_{0} \rvert (w_{U^{\sigma,-}_{0}})^{-\varepsilon} \sum_{i=0}^N (c_1 \varepsilon)^{i} 2^{-(n+p)(1+\varepsilon) i}  \xi^{\varepsilon i} M^i .
\end{align*}
We estimate the sum by
\begin{align*}
\sum_{i=0}^N (c_1 \varepsilon)^{i} 2^{-(n+p)(1+\varepsilon) i}  \xi^{\varepsilon i} M^i &= \sum_{i=0}^N \bigl( c_1 \varepsilon 2^{-(n+p)(1+\varepsilon)}  \xi^{\varepsilon} M \bigr)^i \\
&\leq \frac{1}{1- c_1 \varepsilon 2^{-(n+p)(1+\varepsilon)}  \xi^{\varepsilon} M} ,
\end{align*}
whenever $\varepsilon$ is small enough, for example, $\varepsilon < 2^{n+p} /(c_1 \xi M)$.
Then it holds that
\begin{align*}
\int_{R^+_0(\gamma)} w^{-\varepsilon}
&\leq \frac{c_0 c_2^\varepsilon}{1- c_1 \varepsilon 2^{-(n+p)(1+\varepsilon)}  \xi^{\varepsilon} M} \lvert U^{\sigma,-}_{0} \rvert (w_{U^{\sigma,-}_{0}})^{-\varepsilon}
\end{align*}
for small enough $\varepsilon$.
We conclude that
\begin{equation}
\label{eq:qualiproof_muckenhoupt-estimate}
\dashint_{U^{\sigma,-}_{0}} w \biggl( \dashint_{R^+_0(\gamma)} w^{-\varepsilon} \biggr)^\frac{1}{\varepsilon} \leq c_3 ,
\end{equation}
where
\[
c_3^\varepsilon = \frac{c_0 c_2^\varepsilon}{1- c_1 \varepsilon 2^{-(n+p)(1+\varepsilon)}  \xi^{\varepsilon} M} .
\]

We have shown that~\eqref{eq:qualiproof_muckenhoupt-estimate} holds for weights satisfying that $w^{-1}$ is bounded. 
For general $w$, we consider truncations $\max\{w,1/k\}$, $ k \in \mathbb N $, and apply~\eqref{eq:qualiproof_muckenhoupt-estimate} with Fatou's lemma as $k\to\infty$.
Hence,~\eqref{eq:qualiproof_muckenhoupt-estimate} holds for general weights as well.
Let $q = 1+1/\varepsilon$.
Applying Theorem~\ref{thm:timelagchange}, we conclude that $w \in A^+_q(\gamma)$.
This completes the proof.
\end{proof}

\subsection{Sublevel measure condition}

We show $(iv)\Leftrightarrow (v)$ in Theorem~\ref{thm:Ainfty}.

\begin{theorem}
Let $0\leq\gamma<1$ and $w$ be a weight.
The following conditions are equivalent.
\begin{enumerate}[(i)]
\item There exist $0<\alpha,\beta <1$ such that for every parabolic rectangle $R$ we have
\[
\lvert R^+(\gamma) \cap \{ w < \beta w_{R^-(\gamma)} \} \rvert < \alpha \lvert R^+(\gamma) \rvert .
\]
\item There exist $0<\alpha, \beta<1$ such that for every parabolic rectangle $R$ and every measurable set $E \subset R^+(\gamma)$ for which $w(E) < \beta w(R^-(\gamma))$ it holds that $\lvert E \rvert < \alpha \lvert R^+(\gamma) \rvert$.
\end{enumerate}
\end{theorem}

\begin{proof}
We first show that $(i)$ implies $(ii)$.
Let $E \subset R^+(\gamma)$ be a measurable set such that $w(E) < \beta' w(R^-(\gamma))$ where $\beta' < (1-\alpha)\beta$.
It follows that
\begin{align*}
\lvert E \rvert &= \lvert E \cap \{w < \beta w_{R^-(\gamma)}\} \rvert + \lvert E \cap \{w \geq \beta w_{R^-(\gamma)}\} \rvert\\
&\leq \alpha \lvert R^+(\gamma) \rvert +\frac{1}{\beta w_{R^-(\gamma)}} w(E) \\
&= \biggl( \alpha + \frac{1}{\beta} \frac{w(E)}{w(R^-(\gamma))} \biggr) \lvert R^+(\gamma) \rvert \\
&< \biggl( \alpha + \frac{\beta'}{\beta} \biggr) \lvert R^+(\gamma) \rvert 
= \alpha' \lvert R^+(\gamma) \rvert ,
\end{align*}
where $\alpha' = \alpha + \frac{\beta'}{\beta} < 1$.

We prove the other direction.
Let $E = R^+(\gamma) \cap \{ w < \beta w_{R^-(\gamma)} \}$. Then
\begin{align*}
w(E) < \beta w_{R^-(\gamma)} \lvert E \rvert \leq \beta w(R^-(\gamma)),
\end{align*}
and thus $(ii)$ implies $\lvert E \vert < \alpha \lvert R^+(\gamma) \rvert$.
\end{proof}

\subsection{Gurov--Reshetnyak condition}

We show $(v)\Leftrightarrow (vi)$ in Theorem~\ref{thm:Ainfty}.

\begin{theorem} 
Let $R\subset\mathbb{R}^{n+1}$ be a parabolic rectangle,
$0\leq\gamma<1$ and $w$ be a weight.
\begin{enumerate}[(i)]
\item 
Assume that there exists $0 < \varepsilon < 1$ such that
\[
\dashint_{R^+(\gamma)} (w - w_{R^-(\gamma)})^- \leq \varepsilon w_{R^-(\gamma)} .
\]
Then for $\varepsilon < \lambda < 1 $ we have
\[
\lvert R^+(\gamma) \cap \{ w<(1- \tfrac{\varepsilon}{\lambda} ) w_{R^-(\gamma)} \} \rvert < \lambda \lvert R^+(\gamma) \rvert .
\]
\item
Assume that there exist $0<\alpha,\beta <1$ such that
\[
\lvert R^+(\gamma) \cap \{ w < \beta w_{R^-(\gamma)} \} \rvert < (1-\alpha) \lvert R^+(\gamma) \rvert .
\]
Then we have
\[
\dashint_{R^+(\gamma)} (w - w_{R^-(\gamma)})^- \leq (1-\alpha \beta) w_{R^-(\gamma)} .
\]
\end{enumerate}
\end{theorem}

\begin{proof}
We first show that $(i)$ holds.
Let $E = R^+(\gamma) \cap \{ w < (1- \tfrac{\varepsilon}{\lambda} ) w_{R^-(\gamma)} \}$.
We obtain
\begin{align*}
\frac{\varepsilon}{\lambda} w_{R^-(\gamma)} &\leq \inf_{E} (w_{R^-(\gamma)} - w) \leq \dashint_{E} (w_{R^-(\gamma)} - w)\\
& \leq \frac{1}{\lvert E \rvert} \int_{R^+(\gamma) \cap \{ w \leq w_{R^-(\gamma)} \}} (w_{R^-(\gamma)} - w) \\
& = \frac{1}{\lvert E \rvert} \int_{R^+(\gamma)} (w_{R^-(\gamma)} - w)^+ \leq \frac{\lvert R^+(\gamma) \rvert}{\lvert E \rvert} \varepsilon w_{R^-(\gamma)} ,
\end{align*}
which implies that $\lvert E \rvert \leq \lambda \lvert R^+(\gamma) \rvert $.

For the other direction, we set $E = R^+(\gamma) \cap \{ w < \beta w_{R^-(\gamma)} \}$
and 
\[
E^c = R^+(\gamma) \setminus E = R^+(\gamma) \cap \{ w \geq \beta w_{R^-(\gamma)} \}.
\]
Then $\lvert E \rvert \leq (1-\alpha) \lvert R^+(\gamma) \rvert$ and it holds that
\begin{align*}
&\dashint_{R^+(\gamma)} (w - w_{R^-(\gamma)})^- = \frac{1}{\lvert R^+(\gamma) \rvert} \int_{ R^+(\gamma) \cap \{ w < w_{R^-(\gamma)} \}} (w_{R^-(\gamma)} - w) \\
&\qquad= \frac{1}{\lvert R^+(\gamma) \rvert} \int_{R^+(\gamma) \cap \{ \beta w_{R^-(\gamma)} \leq w < w_{R^-(\gamma)} \}} (w_{R^-(\gamma)} - w) + \frac{1}{\lvert R^+(\gamma) \rvert} \int_{E} (w_{R^-(\gamma)} - w) \\
&\qquad\leq \frac{1}{\lvert R^+(\gamma) \rvert} (1-\beta) w_{R^-(\gamma)} \lvert E^c \rvert  + \frac{1}{\lvert R^+(\gamma) \rvert} w_{R^-(\gamma)} \lvert E \rvert \\
&\qquad= \frac{1}{\lvert R^+(\gamma) \rvert} w_{R^-(\gamma)} ( (1-\beta) \lvert E^c \rvert + \lvert E \rvert ) \\
&\qquad= \frac{1}{\lvert R^+(\gamma) \rvert} w_{R^-(\gamma)} ( (1-\beta) \lvert R^+(\gamma) \rvert + \beta \lvert E \rvert ) \\
&\qquad\leq \frac{1}{\lvert R^+(\gamma) \rvert} w_{R^-(\gamma)} ( (1-\beta) \lvert R^+(\gamma) \rvert + \beta (1-\alpha) \lvert R^+(\gamma) \rvert ) \\
&\qquad= (1-\alpha \beta) w_{R^-(\gamma)} .
\end{align*}
This completes the proof.
\end{proof}

\section{Parabolic reverse H\"older inequality}
\label{sec:RHI}

In this section, we show that parabolic Muckenhoupt weights satisfy the parabolic reverse H\"older inequality.
The lemma below is the main ingredient in a new proof of \cite[Theorem 5.2]{kinnunenSaariParabolicWeighted}.

\begin{lemma}
\label{lemma:weightmeasure-estimate}
Let $0 < \gamma < 1$.
Assume that there exist $0<\alpha<\tfrac{1}{2}$, $0<\beta<1$ such that for every parabolic rectangle $R$ and every measurable set $E \subset R^+(\gamma)$ for which $w(E) < \beta w(R^-(\gamma))$ it holds that $\lvert E \rvert < \alpha \lvert R^+(\gamma) \rvert$.
Then
for every parabolic rectangle $R \subset \mathbb{R}^{n+1}$ and
$\lambda \geq w_{R^-}$ we have
\[
w( R^{-} \cap \{ w > \lambda \}  )  \leq c \lambda \lvert R \cap \{ w > \beta \lambda \} \rvert ,
\]
where $c$ depends on $n,p,\gamma$ and $\alpha$.

\end{lemma}

\begin{proof}

Let $R_0=R(x_0,t_0,L) = Q(x_0,L) \times (t_0-L^p, t_0+L^p)$.
By considering the function $w(x+x_0,t+t_0)$, we may assume that the center of $R_0$ is the origin, that is, $R_0 = Q(0,L) \times (-L^p, L^p)$.

Denote $S^-_0 = R^-_0 = Q(0,L) \times (-L^p,0) $. The time length of $S^-_0$ is $l_t(S^-_0) = L^p$.
We partition $S^-_0$ by dividing each spatial edge $[-L,L]$ into $2$ equally long intervals and the time interval of $S^-_0$ into $\lceil 2^{p}/(1-\gamma) \rceil$ equally long intervals.
We obtain subrectangles $S^-_1$ of $S^-_0$ with spatial side length 
$l_x(S^-_1)=l_x(S^-_0)/2 = L / 2$ and time length 
\[
l_t(S^-_1)= \frac{L^p}{\lceil 2^{p} / (1-\gamma) \rceil}.
\]
For every $S^-_1$, there exists a unique rectangle $R_1$ with spatial side length $l_x = L / 2$ and time length $l_t = 2 L^p / 2^{p}$
such that $R_1$ has the same bottom as $S^-_1$.
We select those rectangles $S^-_1$ for which 
\[
\lambda < \dashint_{S^-_1} w \dx\dt = \frac{w(S^-_1)}{\lvert S^-_1 \rvert} 
\]
and denote the obtained collection by $\{ S^-_{1,j} \}_j$.
If
\[
\lambda \geq \dashint_{S^-_1} w \dx\dt ,
\]
we subdivide $S^-_1$ by dividing each spatial edge $[-L,L]$ into $2$ equally long intervals.
If
\[
\frac{l_t(S_{1}^-)}{\lfloor 2^{p} \rfloor} \leq \frac{(1-\gamma)L^p}{2^{2p}},
\]
we divide the time interval of $S^-_1$ into $\lfloor 2^{p} \rfloor$ equally long intervals. 
Otherwise, we divide the time interval of $S^-_1$ into $\lceil 2^{p} \rceil$ equally long intervals.
We obtain subrectangles $S^-_2$ of  $S^-_1$ with spatial side length 
$l_x(S^-_2)=l_x(S^-_1)/2$ and time length either 
\[
l_t(S^-_2)=\frac{l_t(S^-_1)}{\lfloor 2^{p} \rfloor} 
\quad\text{or}\quad
l_t(S^-_2)=\frac{l_t(S^-_1)}{\lceil 2^{p} \rceil}.
\]
Select all those subrectangles $S^-_2$ for which
\[
\lambda < \dashint_{S^-_2} w \dx\dt = \frac{w(S^-_2)}{\lvert S^-_2 \rvert} 
\]
to obtain family $\{ S^-_{2,j} \}_j$.
We continue this selection process recursively.
At the $i$th step, we partition unselected rectangles $S^-_{i-1}$ by dividing each spatial side into $2$ equal parts, and if 
\begin{equation}
\label{eq:RHIproof_eq1}
\frac{l_t(S_{i-1}^-)}{\lfloor 2^{p} \rfloor} \leq \frac{(1-\gamma)L^p}{2^{pi}},
\end{equation}
we divide the time interval of $S^-_{i-1}$ into $\lfloor 2^{p} \rfloor$ equal parts. 
If
\begin{equation}
\label{eq:RHIproof_eq2}
\frac{l_t(S_{i-1}^-)}{\lfloor 2^{p} \rfloor} > \frac{(1-\gamma)L^p}{2^{pi}},
\end{equation}
we divide the time interval of $S^-_{i-1}$ into $\lceil 2^{p} \rceil$ equal parts.
We obtain subrectangles $S^-_i$. 
For every $S^-_i$, there exists a unique rectangle $R_i$ with spatial side length $l_x = L / 2^{i}$ and time length $l_t = 2 L^p / 2^{pi}$
such that $R_i$ has the same top as $S^-_i$.
Select those $S^-_i$ for which 
\[
\lambda < \dashint_{S^-_i} w \dx\dt = \frac{w(S^-_i)}{\lvert S^-_i \rvert} 
\]
and denote the obtained collection by $\{ S^-_{i,j} \}_j$.
If 
\[
\lambda \geq \dashint_{S^-_i} w \dx\dt ,
\]
we continue the selection process in $S^-_i$.
This manner we obtain a collection $\{S^-_{i,j} \}_{i,j}$ of pairwise disjoint rectangles.

Observe that if \eqref{eq:RHIproof_eq1} holds, then we have
\[
l_t(S_i^-) = \frac{l_t(S^-_{i-1})}{\lfloor 2^{p} \rfloor} \leq \frac{(1-\gamma)L^p}{2^{pi}}.
\]
On the other hand, if \eqref{eq:RHIproof_eq2} holds, then
\[
l_t(S_i^-) = \frac{l_t(S^-_{i-1})}{\lceil 2^{p} \rceil} \leq \frac{l_t(S^-_{i-1})}{2^{p}} \leq \dots \leq \frac{(1-\gamma)L^p}{2^{pi}} .
\]
Hence, we get an upper bound for the time length
\[
l_t(S_i^-) \leq \frac{(1-\gamma)L^p}{2^{pi}}
\]
for every $S_i^-$.

Suppose that \eqref{eq:RHIproof_eq2} is satisfied at the $i$th step, that is,
\[
\frac{l_t(S_{i-1}^-)}{\lfloor 2^{p} \rfloor} > \frac{(1-\gamma)L^p}{2^{pi}}.
\]
Then we have a lower bound for the time length of $S_i^-$, since
\begin{align*}
l_t(S^-_i) = \frac{l_t(S_{i-1}^-)}{\lceil 2^{p} \rceil} > \frac{\lfloor 2^{p} \rfloor}{\lceil 2^{p} \rceil} \frac{(1-\gamma)L^p}{2^{pi}} > \frac{1}{2} \frac{(1-\gamma)L^p}{2^{pi}} .
\end{align*}
On the other hand, if \eqref{eq:RHIproof_eq1} is satisfied, then
\[
l_t(S^-_i) = \frac{l_t(S_{i-1}^-)}{\lfloor 2^{p} \rfloor} \geq \frac{l_t(S_{i-1}^-)}{ 2^{p}}.
\]
There are two alternatives at the $i$th step of the selection process:
Either \eqref{eq:RHIproof_eq2} has not yet been satisfied at the earlier steps or \eqref{eq:RHIproof_eq2} has been satisfied at a step $i'$ with $i'\le i$.
In the first case, we have
\begin{align*}
l_t(S^-_i) \geq \frac{l_t(S_{i-1}^-)}{ 2^{p}} \geq \frac{l_t(S_{i-2}^-)}{ 2^{2p}} \geq \dots \geq \frac{l_t(S_{0}^-)}{ 2^{pi}} > \frac{1}{2} \frac{(1-\gamma)L^p}{ 2^{pi}} .
\end{align*}
In the second case, we obtain the same estimate
\begin{align*}
l_t(S^-_i) \geq \frac{l_t(S_{i-1}^-)}{ 2^{p}} \geq \dots \geq \frac{l_t(S_{i'}^-)}{ 2^{p(i-i')}} > \frac{1}{2} \frac{(1-\gamma)L^p}{ 2^{pi}}
\end{align*}
by using the lower bound for $S_{i'}^-$.
Thus, we have 
\[
\frac{1}{2} \frac{(1-\gamma)L^p}{2^{pi}} \leq l_t(S^-_i) \leq \frac{(1-\gamma)L^p}{2^{pi}}
\]
for every $S^-_i$.
By using the lower bound for the time length of $S^-_i$,
we observe that
\begin{align*}
l_t(R_1) - l_t(S^-_1) 
&\leq \frac{2 L^p}{2^{p}} - \frac{1}{2} \frac{(1-\gamma)L^p}{2^{p}} \\
&= \frac{1}{2} \frac{L^p}{2^{p}} (3+ \gamma) 
\leq L^p = l_t(R^+_{0}) .
\end{align*}
This implies $R_i \subset R_0$
for every $i\in \mathbb{N}$.
By the construction of the subrectangles $S^-_i$, we have
\begin{equation}
\label{eq:RHIproof_measure1_1}
\lvert S^-_{0} \rvert = 2^{n} \bigg\lceil \frac{2^{p}}{1-\gamma} \bigg\rceil \lvert S^-_1 \rvert 
\end{equation}
and
\begin{equation}
\label{eq:RHIproof_measure1_2}
2^{n} \lfloor 2^{p} \rfloor \lvert S^-_i \rvert 
\leq \lvert S^-_{i-1} \rvert \leq 2^{n} \lceil 2^{p} \rceil \lvert S^-_i \rvert 
\end{equation}
for $i \geq 2$.

We have a collection $\{ S^-_{i,j} \}_{i,j}$ of pairwise disjoint rectangles. 
However, the rectangles in the corresponding collection $\{ S^+_{i,j} \}_{i,j}$ may overlap. 
Thus, we replace it by a maximal subfamily $\{ \widetilde{S}^+_{i,j} \}_{i,j}$ of pairwise disjoint rectangles, which is constructed in the following way.
At the first step, choose $\{ S^+_{1,j} \}_{j}$ and denote it $\{ \widetilde{S}^+_{1,j} \}_j$. 
Then consider the collection $\{ S^+_{2,j} \}_{j}$ where each $S^+_{2,j}$ either intersects some $\widetilde{S}^+_{1,j}$ or does not intersect any $\widetilde{S}^+_{1,j}$. 
Select the rectangles $S^+_{2,j}$, that do not intersect any $\widetilde{S}^+_{1,j}$, and denote the obtained collection $\{ \widetilde{S}^+_{2,j} \}_j$.
At the $i$th step, choose those $S^+_{i,j}$ that do not intersect any previously selected $\widetilde{S}^+_{i',j}$, $i' < i$.
Hence, we obtain a collection $\{ \widetilde{S}^+_{i,j} \}_{i,j}$ of pairwise disjoint rectangles.
Observe that for every $i,j$ there exists $i',j'$ with $i' < i$ such that
\begin{equation}
\label{eq:RHIproof_plussubset}
\text{pr}_x(S^+_{i,j}) \subset \text{pr}_x(\widetilde{S}^+_{i',j'}) \quad \text{and} \quad \text{pr}_t(S^+_{i,j}) \subset 3 \text{pr}_t(\widetilde{S}^+_{i',j'}) .
\end{equation}
Here pr$_x$ denotes the projection to $\mathbb R^n$ and pr$_t$ denotes the projection to the time axis.

Rename $\{ S^-_{i,j} \}_{i,j}$, $\{ \widetilde{S}^+_{i,j} \}_{i,j}$ as $\{ S^-_{i} \}_{i}$, $\{ \widetilde{S}^+_{j} \}_j$.
Note that $S^-_i$ is spatially contained in $S^+_i$, that is,  $\text{pr}_xS^-_i\subset \text{pr}_x S^+_i$.
In the time direction, we have
\begin{equation}
\label{eq:RHIproof_minusplussubset}
\text{pr}_t(S^-_i) \subset \text{pr}_t(R_i) 
\subset \frac{7 + \gamma}{1-\gamma} \text{pr}_t(S^+_i) ,
\end{equation}
since
\[
\biggl( \frac{7 + \gamma}{1-\gamma} + 1 \biggr) \frac{l_t(S^+_i)}{2} \geq  \frac{8}{1-\gamma}  \frac{(1-\gamma)L^p}{2^{pi+2}} = \frac{2L^p}{2^{pi}} = l_t(R_i) .
\]
Therefore, by~\eqref{eq:RHIproof_plussubset} and~\eqref{eq:RHIproof_minusplussubset}, it holds that
\begin{equation}
\label{eq:RHIproof_subsetcubes}
\sum_i \lvert S^-_i \rvert = \Big\lvert \bigcup_i S^-_i \Big\rvert \leq c_1 \sum_j \lvert \widetilde{S}^+_j \rvert ,
\end{equation}
with $c_1 = 3 \frac{7+\gamma}{1-\gamma}$.

If $(x,t) \in S^-_0 \setminus \bigcup_i S^-_i$, then there exists a sequence of subrectangles $S^-_l$ containing $(x,t)$ such that 
\[
\frac{w(S^-_l)}{\lvert S^-_l \rvert} = \dashint_{S^-_l} w \dx\dt \leq \lambda
\]
and $\lvert S^-_l \rvert \to 0$ as $l \to \infty$.
The Lebesgue differentiation theorem~\cite[Lemma~2.3]{KinnunenMyyryYang2022}
implies that $w(x,t) \leq \lambda$
for almost every $(x,t) \in S^-_0 \setminus \bigcup_i S^-_i$.
It follows that
\[
S^-_0 \cap \{ w > \lambda \} \subset \bigcup_i S^-_i
\]
up to a set of measure zero.

By~\eqref{eq:RHIproof_measure1_1},~\eqref{eq:RHIproof_measure1_2} and~\eqref{eq:RHIproof_subsetcubes}, we obtain
\begin{align*}
w( S^-_0 \cap \{ w > \lambda \} ) &\leq \sum_i w( S^-_i )\leq \sum_i w( S^-_{i-1} )\leq \lambda \sum_i \lvert S^-_{i-1} \rvert \\
&\leq 2^{n} \bigg\lceil \frac{2^{p}}{1-\gamma} \bigg\rceil \lambda \sum_i \lvert S^-_i \rvert 
\leq c_1 \frac{2^{n+p+1}}{1-\gamma} \lambda \sum_j \lvert \widetilde{S}^+_j \rvert .
\end{align*}
We have
\begin{align*}
w( \widetilde{S}^+_j \cap \{ w < \beta w_{\widetilde{S}^-_j} \} ) \leq \beta w_{\widetilde{S}^-_j} \lvert \widetilde{S}^+_j \rvert = \beta w(\widetilde{S}^-_j) \leq \beta w(\widetilde{R}^-_j(\gamma)) .
\end{align*}
Then by the assumption (qualitative measure condition) it holds that
\[
\lvert \widetilde{S}^+_j \cap \{ w < \beta w_{\widetilde{S}^-_j} \} \rvert < \alpha \lvert \widetilde{R}^+_j(\gamma) \rvert \leq 2 \alpha \lvert \widetilde{S}^+_j \rvert
\]
from which we obtain
\[
(1- 2\alpha) \lvert \widetilde{S}^+_j \rvert < \lvert \widetilde{S}^+_j \cap \{ w \geq \beta w_{\widetilde{S}^-_j} \} \rvert \leq \lvert \widetilde{S}^+_j \cap \{ w > \beta \lambda \} \rvert   ,
\]
since $w_{\widetilde{S}^-_j} > \lambda $.
Thus, we can conclude that
\begin{align*}
w( S^-_0 \cap \{ w > \lambda \} ) &\leq c \lambda \sum_j \lvert \widetilde{S}^+_j \cap \{ w > \beta \lambda \} \rvert \leq c \lambda \lvert R_0 \cap \{ w > \beta \lambda \} \rvert ,
\end{align*}
where
\[
c = 
c_1 \frac{2^{n+p+1}}{1-\gamma} \frac{1}{1- 2\alpha}.
\]
\end{proof}

With the previous lemma, we show that the parabolic Muckenhoupt condition implies the parabolic reverse H\"older inequality, compare with \cite[Theorem 5.2]{kinnunenSaariParabolicWeighted}.

\begin{theorem}
\label{thm:RHI}
Let $1\leq q<\infty$, $0<\gamma<1$ and $w \in A_q^+(\gamma)$.
Then there exists $\varepsilon>0$ and $c>0$ such that
\[
\biggl( \dashint_{R^-} w^{1+\varepsilon} \biggr)^\frac{1}{1+\varepsilon} \leq c \dashint_{R^+} w
\]
for every parabolic rectangle $R \subset \mathbb R^{n+1}$.
\end{theorem}

\begin{proof}

By Theorem~\ref{thm:Ainfty}, we see that the assumptions of Lemma~\ref{lemma:weightmeasure-estimate} are satisfied and thus it can be applied.
We prove the claim first for bounded functions. Thus, assume that $w$ is bounded.
Let $R \subset \mathbb R^{n+1}$ be a parabolic rectangle.
Let $\varepsilon>0$ to be chosen later.
Applying Cavalieri's principle with Lemma~\ref{lemma:weightmeasure-estimate}, we obtain
\begin{align*}
\int_{R^-} w^{1+\varepsilon} &= \varepsilon \int_0^\infty \lambda^{\varepsilon-1} w( R^{-} \cap \{ w > \lambda \}  ) \dla \\
&= \varepsilon \int_0^{w_{R^-}} \lambda^{\varepsilon-1} w( R^{-} \cap \{ w > \lambda \}  ) \dla + \varepsilon \int_{w_{R^-}}^\infty \lambda^{\varepsilon-1} w( R^{-} \cap \{ w > \lambda \}  ) \dla \\
&\leq w( R^{-} ) \varepsilon \int_0^{w_{R^-}} \lambda^{\varepsilon-1} \dla + c \varepsilon \int_{w_{R^-}}^\infty \lambda^{\varepsilon} \lvert R \cap \{ w > \beta \lambda \} \rvert \dla \\
&\leq w( R^{-} ) w_{R^-}^\varepsilon + c \beta^{-1-\varepsilon} \varepsilon \int_0^\infty \lambda^{\varepsilon} \lvert R \cap \{ w > \lambda \} \rvert \dla \\
&\leq \lvert R^{-} \rvert w_{R^-}^{1+\varepsilon} + c \beta^{-1-\varepsilon} \frac{\varepsilon}{1+\varepsilon} \int_{R} w^{1+\varepsilon} .
\end{align*}
By choosing $\varepsilon>0$ to be small enough, we can absorb the integral over $R^-$ of the second term to the left-hand side to get
\begin{align*}
\biggl( 1- \frac{c}{\beta^{1+\varepsilon}} \frac{\varepsilon}{1+\varepsilon} \biggr) \int_{R^-} w^{1+\varepsilon} \leq \lvert R^{-} \rvert w_{R^-}^{1+\varepsilon} + \frac{c}{\beta^{1+\varepsilon}} \frac{\varepsilon}{1+\varepsilon} \int_{R^+} w^{1+\varepsilon} .
\end{align*}
Hence, we have
\begin{equation}
\label{eq:RHIproof_cavalieri_iteration}
\int_{R^-} w^{1+\varepsilon} \leq c_0 \lvert R^{-} \rvert w_{R^-}^{1+\varepsilon} + c_1 \varepsilon \int_{R^+} w^{1+\varepsilon} ,
\end{equation}
where
\[
c_0 = \frac{1+\varepsilon }{1-(c \beta^{-1-\varepsilon} -1) \varepsilon} \quad \text{and} \quad c_1 = \frac{c \beta^{-1-\varepsilon} }{1-(c \beta^{-1-\varepsilon} -1) \varepsilon} .
\]

Fix $R_0 = Q(x_0,L) \times (t_0 - L^p, t_0 + L^p) \subset \mathbb R^{n+1}$.
We cover $R^-_0$ by $M = 2^{n+1}$ rectangles $R^-_{1,j}$ with spatial side length $l_x = L/2^{1/p}$ and time length $l_t = L^p / 2$. This can be done by dividing each spatial edge of $R^-_0$ into two equally long intervals that may overlap each other, and the time interval of $R^-_0$ into two equally long pairwise disjoint intervals.
Observe that the overlap of $R^-_{1,j}$ is bounded by $M/2 = 2^n$.
Then consider $R^+_{1,j}$ and cover it in the same way as before by $M$ rectangles $R^-_{2,j}$ with spatial side length $l_x = L/2^{2/p}$ and time length $l_t = L^p / 2^2$.
At the $i$th step, cover $R^+_{i-1,j}$ by $M$ rectangles $R^-_{i,j}$ with spatial side length $l_x = L/2^{i/p}$ and time length $l_t = L^p / 2^i$ such that their overlap is bounded by $M/2$.
Note that every $R_{i,j}$ is contained in $R_0$.
Then iterating~\eqref{eq:RHIproof_cavalieri_iteration} we obtain
\begin{align*}
\int_{R^-_0} w^{1+\varepsilon} &\leq \sum_{j=1}^{M} \int_{R^-_{1,j}} w^{1+\varepsilon}  \leq \sum_{j=1}^{M} c_0 \lvert R^{-}_{1,j} \rvert w_{R^-_{1,j}}^{1+\varepsilon} + \sum_{j=1}^{M} c_1 \varepsilon \int_{R^+_{1,j}} w^{1+\varepsilon} \\
&\leq c_0 \sum_{j=1}^M \lvert R^{-}_{1,j} \rvert w_{R^-_{1,j}}^{1+\varepsilon} + c_1 \varepsilon \sum_{j=1}^{M^2} \int_{R^-_{2,j}} w^{1+\varepsilon} \\
&\leq c_0 \sum_{j=1}^M \lvert R^{-}_{1,j} \rvert w_{R^-_{1,j}}^{1+\varepsilon} + c_1 \varepsilon \sum_{j=1}^{M^2} \biggl( c_0 \lvert R^-_{2,j}\rvert w_{R^-_{2,j}}^{1+\varepsilon} + c_1 \varepsilon \int_{R^+_{2,j}} w^{1+\varepsilon} \biggr) \\
&= c_0 \sum_{j=1}^M \lvert R^{-}_{1,j} \rvert w_{R^-_{1,j}}^{1+\varepsilon} + c_0 c_1 \varepsilon \sum_{j=1}^{M^2} \lvert R^-_{2,j}\rvert w_{R^-_{2,j}}^{1+\varepsilon} + (c_1 \varepsilon)^2 \sum_{j=1}^{M^2} \int_{R^+_{2,j}} w^{1+\varepsilon} \\
&\leq c_0 \sum_{i=1}^N \biggl( (c_1 \varepsilon)^{i-1} \sum_{j=1}^{M^i} \lvert R^-_{i,j} \rvert w^{1+\varepsilon}_{R^-_{i,j}} \biggr) + (c_1 \varepsilon)^N \sum_{j=1}^{M^N} \int_{R^+_{N,j}} w^{1+\varepsilon} \\
&\leq c_0 \sum_{i=1}^N \biggl( (c_1 \varepsilon)^{i-1} \sum_{j=1}^{M^i} \lvert R^-_{i,j} \rvert w^{1+\varepsilon}_{R^-_{i,j}} \biggr) + \biggl(c_1 \varepsilon \frac M2\biggr)^N \int_{R_0} w^{1+\varepsilon} \\
&= I + II .
\end{align*}
We observe that $II$ tends to zero if $\varepsilon < \tfrac{2}{c_1 M} = \tfrac{1}{c_1 2^n}$ as $N \to \infty$.
Since
\[
\lvert R^-_{i,j} \rvert^{-\varepsilon} = L^{-(n+p)\varepsilon} 2^{(\tfrac{n}{p}+1)i \varepsilon } =  2^{1+\varepsilon} L^{n+p} 2^{(\tfrac{n}{p}+1)i \varepsilon } \lvert R_{0} \rvert^{-(1+\varepsilon)},
\]
for the inner sum of the first term $I$ we have 
\begin{align*}
\sum_{j=1}^{M^i} \lvert R^-_{i,j} \rvert w^{1+\varepsilon}_{R^-_{i,j}} = \sum_{j=1}^{M^i} \lvert R^-_{i,j} \rvert^{-\varepsilon} \biggl( \int_{R^-_{i,j}} w \biggr)^{1+\varepsilon} \leq 2^{1+\varepsilon} L^{n+p} 2^{(\tfrac{n}{p}+1)i \varepsilon } \biggl(\frac M2\biggr)^i w^{1+\varepsilon}_{R_0} .
\end{align*}
Thus, it follows that
\begin{align*}
I \leq c_0 2^{1+\varepsilon} L^{n+p} w^{1+\varepsilon}_{R_0}  \sum_{i=1}^N  (c_1 \varepsilon)^{i-1} 2^{(\tfrac{n}{p}+1)i \varepsilon }  \biggl(\frac M2\biggr)^i .
\end{align*}
We estimate the sum by
\begin{align*}
\sum_{i=1}^N  (c_1 \varepsilon)^{i-1} 2^{(\tfrac{n}{p}+1)i \varepsilon } \biggl(\frac M2\biggr)^i
&= 2^{(\tfrac{n}{p}+1) \varepsilon } \frac{M}{2} \sum_{i=0}^{N-1} \biggl( c_1 \varepsilon 2^{(\tfrac{n}{p}+1) \varepsilon } \frac{M}{2} \biggr)^i \\
&\leq 2^{(\tfrac{n}{p}+1) \varepsilon } \frac{M}{2} \frac{1}{1-c_1 \varepsilon 2^{(\tfrac{n}{p}+1) \varepsilon } \frac{M}{2} } \\
&= \frac{2^{(\tfrac{n}{p}+1) \varepsilon +n }}{1-c_1 \varepsilon 2^{(\tfrac{n}{p}+1) \varepsilon + n } } ,
\end{align*}
whenever $\varepsilon$ is small enough, for example
\[
\varepsilon < \frac{1}{c_1 2^{\tfrac{n}{p}} M } = \frac{1}{c_1 2^{\tfrac{n}{p}+n+1} } .
\]
Then it holds that
\begin{align*}
\int_{R^-_0} w^{1+\varepsilon} &\leq c_0 2^{1+\varepsilon} L^{n+p} w^{1+\varepsilon}_{R_0} \frac{ 2^{(\tfrac{n}{p}+1) \varepsilon +n } }{1- c_1 \varepsilon 2^{(\tfrac{n}{p}+1) \varepsilon +n } }
\end{align*}
for small enough $\varepsilon$.
Since $w_{R^-_0} \leq C w_{R^+_0}$ for some $C=C(\gamma, [w]_{A^+_q(\gamma)})$, see the proof of Theorem~\ref{thm:timelagchange},
we conclude that
\begin{align*}
\biggl( \dashint_{R^-_0} w^{1+\varepsilon} \biggr)^\frac{1}{1+\varepsilon} &\leq c_2 \dashint_{R_0} w = \frac{c_2}{2} \dashint_{R^-_0} w + \frac{c_2 }{2} \dashint_{R^+_0} w \leq \frac{c_2 }{2} (C+1) \dashint_{R^+_0} w ,
\end{align*}
where
\[
c_2 = 2 \Biggl( c_0 \frac{ 2^{(\tfrac{n}{p}+1) \varepsilon +n } }{1- c_1 \varepsilon 2^{(\tfrac{n}{p}+1) \varepsilon +n}} \Biggr)^\frac{1}{1+\varepsilon} .
\]
Hence, the claim holds for bounded functions. For unbounded $w$, we consider truncations $\min\{w,k\}$, $ k \in \mathbb N $, and apply the claim with the monotone convergence theorem as $k \to \infty$.
This completes the proof.
\end{proof}

\begin{corollary}
\label{cor:RHIlagged}
Let $1\leq q<\infty$, $0<\gamma<1$ and $w \in A_q^+(\gamma)$.
Let $0\leq \alpha<1$ and $\tau>0$.
Then there exist $\varepsilon>0$ and $c>0$ such that
\[
\biggl( \dashint_{R^-(\alpha)} w^{1+\varepsilon} \biggr)^\frac{1}{1+\varepsilon} \leq c \dashint_{S^+(\alpha)} w
\]
for every parabolic rectangle $R = R(x,t,L) \subset \mathbb R^{n+1}$, where 
$S^+(\alpha) = R^-(\alpha) + (0, \tau (1-\alpha) L^p)$.
\end{corollary}

\begin{proof}
Let $R \subset \mathbb R^{n+1}$ be a parabolic rectangle with side length $L$.
Choose $N\in\mathbb{N}$ and $0<\beta\leq1$ such that
$\tau = N + \beta$.
Denote $S_0^+(\alpha) = R^-(\alpha) + (0, \beta (1-\alpha) L^p)$
and $S_k^+(\alpha) = R^-(\alpha) + (0, (k+\beta) (1-\alpha) L^p)$ for $k=1,\dots,N$.
Note that $S_N^+(\alpha) = S^+(\alpha)$.

Denote $\rho = \beta^{1/p} (1-\alpha)^{1/p}$.
We partition $R^-(\alpha)$ into $M = \lceil \rho^{-1} \rceil^n \lceil \rho^{-p} \rceil$ subrectangles $R^-_{i}$ with spatial side length $\rho L$ and time length $\rho^p L^p$ such that the overlap of $\{R^-_{i}\}_i$ is bounded by $2^{n+1}$. 
This can be done by dividing each spatial edge of $R^-(\alpha)$ into $\lceil \rho^{-1} \rceil$ equally long 
subintervals with an overlap bounded by $2$, and the time interval of $R^-(\alpha)$ into $\lceil \rho^{-p} \rceil$ equally long subintervals with an overlap bounded by $2$.
We observe that every $R^{+}_i$ is contained in $S^+_0(\alpha)$.
Then
Theorem~\ref{thm:RHI} implies that there exists a constant $C_1$ such that
\begin{align*}
\biggl( \dashint_{R^-(\alpha)} w^{1+\varepsilon} \biggr)^\frac{1}{1+\varepsilon} 
&\leq 
\Biggl( \sum_i \frac{\lvert R^-_i \rvert}{\lvert R^-(\alpha) \rvert} \dashint_{R^-_i} w^{1+\varepsilon} \Biggr)^\frac{1}{1+\varepsilon} \\
&\leq
\biggl( \frac{\rho^{n+p}}{1-\alpha} \biggr)^\frac{1}{1+\varepsilon}
\sum_i \biggl( \dashint_{R^-_i} w^{1+\varepsilon} \biggr)^\frac{1}{1+\varepsilon} \\
&\leq 
\bigl( \beta^{\frac{n}{p}+1} (1-\alpha)^\frac{n}{p} \bigr)^\frac{1}{1+\varepsilon} 
C_1 \sum_i \dashint_{R^{+}_i} w \\
&= 
\bigl( \beta^{\frac{n}{p}+1} (1-\alpha)^\frac{n}{p} \bigr)^\frac{1}{1+\varepsilon} 
C_1 \sum_i \frac{\lvert S^+_0(\alpha) \rvert}{\lvert R^{+}_i \rvert} \frac{1}{\lvert S^+_0(\alpha) \rvert} \int_{R^{+}_i} w \\
&\leq
\bigl( \beta^{\frac{n}{p}+1} (1-\alpha)^\frac{n}{p} \bigr)^{\frac{1}{1+\varepsilon}-1}
C_1
2^{n+1} \dashint_{S^{+}_0(\alpha)} w \\
&= 
C_2
\dashint_{S^{+}_0(\alpha)} w ,
\end{align*}
where $C_2 = \bigl( \beta^{\frac{n}{p}+1} (1-\alpha)^\frac{n}{p} \bigr)^{-\frac{\varepsilon}{1+\varepsilon}} C_1 2^{n+1}$.

By iterating the previous argument and H\"older's inequality, we obtain
\begin{align*}
\dashint_{S^{+}_0(\alpha)} w &\leq \biggl( \dashint_{S^{+}_0(\alpha)} w^{1+\varepsilon} \biggr)^\frac{1}{1+\varepsilon} 
\leq \frac{C_1 2^{n+1}}{(1-\alpha)^{\frac{n }{p} 
\frac{\varepsilon}{1+\varepsilon} 
}}
\dashint_{S^{+}_1(\alpha)} w \\
&\leq C_3^N \dashint_{S^{+}_N(\alpha)} w
\leq C_3^\tau \dashint_{S^{+}(\alpha)} w ,
\end{align*}
where $C_3 = (1-\alpha)^{-\frac{n }{p} 
\frac{\varepsilon}{1+\varepsilon} 
} C_1 2^{n+1}$.
Thus, we conclude that
\[
\biggl( \dashint_{R^-(\alpha)} w^{1+\varepsilon} \biggr)^\frac{1}{1+\varepsilon} 
\leq 
C_2
\dashint_{S^{+}_0(\alpha)} w
\leq
C_2
C_3^\tau \dashint_{S^{+}(\alpha)} w .
\]
\end{proof}

The parabolic reverse H\"older inequality implies the following self-improving property for parabolic Muckenhoupt weights.

\begin{corollary}
\label{cor:Aqselfimprove}
Let $1<q<\infty$, $0<\gamma<1$ and $w \in A^+_{q}(\gamma)$. Then there exists $\varepsilon=\varepsilon(n,p,q,\gamma,[w]_{A^+_{q}}) >0$ such that $w \in A^+_{q-\varepsilon}(\gamma)$ with $[w]_{A^+_{q-\varepsilon}}$ depending on 
$n,p,q,\gamma$ and $[w]_{A^+_{q}}$.
\end{corollary}

\begin{proof}
We have $w^\frac{1}{1-q}  \in A_{q'}^-(\gamma)$.
Corollary~\ref{cor:RHIlagged} implies that there exist $\delta>0$ and $c>0$ such that
\begin{align*}
\biggl( \dashint_{R^+(\gamma)}  
w^\frac{1+\delta}{1-q}
\biggr)^\frac{1}{1+\delta} \leq c \dashint_{R^-(\gamma)} w^\frac{1}{1-q} 
\end{align*}
for every parabolic rectangle $R\subset\mathbb{R}^{n+1}$.
Let $\varepsilon = \frac{\delta}{1+\delta}(q-1)$. 
By the previous estimate and the parabolic Muckenhoupt condition, we obtain
\begin{align*}
\biggl( \dashint_{R^+(\gamma)} w^\frac{1}{1-(q-\varepsilon)} \biggr)^{q-\varepsilon-1}
&=
\biggl( \dashint_{R^+(\gamma)} w^\frac{1+\delta}{1-q} \biggr)^\frac{q-1}{1+\delta}
\leq
c 
\biggl( \dashint_{R^-(\gamma)} w^\frac{1}{1-q} \biggr)^{q-1}\\
&\leq
c [w]_{A^+_{q}(\gamma)}
\biggl( \dashint_{R^{--}(\gamma)} w\biggr)^{-1}
\end{align*}
for every parabolic rectangle $R\subset\mathbb{R}^{n+1}$,
where $R^{--}(\gamma)$ denotes the lower part of $R^{-}(\gamma)$.
Thus, by Theorem~\ref{thm:timelagchange} it follows that $w \in A^+_{q-\varepsilon}(\gamma)$.
\end{proof}

\section{Weighted norm inequalities for parabolic maximal functions}
\label{sec:parmaxfct}

The next theorem is a weak type estimate for the uncentered parabolic maximal function.
For the corresponding result for the centered parabolic maximal function, see~\cite{kinnunenSaariParabolicWeighted,MaHeYan2023,KinnunenMyyryYangZhu2022}.
The proof is based on~\cite{ForzaniMartinreyesOmbrosi2011}.

\begin{theorem}
\label{thm:weaktype}
Let $1\leq q<\infty$, $0\leq\gamma<1$ and $w$ be a weight.
Then $w \in A^+_{q}(\gamma)$ if and only if
there exists a constant $C=C(n,p,q,\gamma, [w]_{A^+_{q}(\gamma)})$ such that
\[
w( \{ M^{\gamma+} f > \lambda \}) \leq 
\frac{C}{\lambda^q} \int_{\mathbb{R}^{n+1}} \lvert f \rvert^q \, w
\]
for every $\lambda>0$ and $f\in L^1_{\mathrm{loc}}(\mathbb{R}^{n+1})$.
\end{theorem}

\begin{proof}
Assume that the weak type estimate holds.
We observe that it is enough to prove the claim for $w$ that is bounded from below.
Let $f= w^{1-q'} \chi_{R^+(\gamma)} $
and 
$0<\lambda<(w^{1-q'})_{R^+(\gamma)}$.
Since
\[
\dashint_{R^+(\gamma)} w^{1-q'} \leq M^{\gamma+} (w^{1-q'} \chi_{R^+(\gamma)})(z)
\]
for every $z\in R^-(\gamma)$, we have
\begin{align*}
w(R^-(\gamma)) &\leq w( \{ M^{\gamma+} (w^{1-q'} \chi_{R^+(\gamma)}) > \lambda \}) \\
&\leq \frac{C}{\lambda^q} \int_{\mathbb{R}^{n+1}} \lvert w^{1-q'} \chi_{R^+(\gamma)} \rvert^q \, w
= \frac{C}{\lambda^q} \int_{R^+(\gamma)} w^{1-q'} .
\end{align*}
By letting $\lambda\to (w^{1-q'})_{R^+(\gamma)}$, we obtain
\[
w(R^-(\gamma)) \biggl( \dashint_{R^+(\gamma)} w^{1-q'} \biggr)^q
\leq C \int_{R^+(\gamma)} w^{1-q'},
\]
that is
\[
\biggl( \dashint_{R^-(\gamma)} w \biggr) \biggl( \dashint_{R^+(\gamma)} w^{1-q'} \biggr)^{q-1}
\leq C .
\]
By taking supremum over all parabolic rectangles $R\subset\mathbb{R}^{n+1}$, we conclude that $w\in A^+_q(\gamma)$.

To prove the other direction, we make the following assumptions.
First we may assume $f$ to be bounded and compactly supported function.
We note that it suffices to prove the claim for the following restricted maximal function
\[
M^{\gamma+}_\xi f(x,t) = \sup_{\substack{ R^-(\gamma) \ni (x,t) \\ l(R)\geq \xi } } \dashint_{R^+(\gamma)} \lvert f \rvert, \quad \xi<1,
\]
and then let $\xi \to 0$ to obtain the conclusion.
Next we can reduce the statement to showing that
\begin{align*}
w( \{ \lambda < M^{\gamma+}_\xi f \leq 2\lambda \}) \leq \frac{C}{\lambda^q} \int_{\mathbb{R}^{n+1}} \lvert f \rvert^q \, w ,
\end{align*}
since by summing up we get
\begin{align*}
w( \{ M^{\gamma+}_\xi f > \lambda \}) &\leq
\sum_{k=0}^\infty w( \{ 2^{k} \lambda < M^{\gamma+}_\xi f \leq 2^{k+1} \lambda \}) \\
&\leq
\sum_{k=0}^\infty
\frac{C}{2^{kq} \lambda^q} \int_{\mathbb{R}^{n+1}} \lvert f \rvert^q \, w 
\leq
\frac{2C}{\lambda^q} \int_{\mathbb{R}^{n+1}} \lvert f \rvert^q \, w .
\end{align*}
Moreover, it is sufficient to prove that for $L>0$ we have
\begin{align*}
w( B(0,L) \cap \{ \lambda < M^{\gamma+}_\xi f \leq 2\lambda \}) \leq \frac{C}{\lambda^q} \int_{\mathbb{R}^{n+1}} \lvert f \rvert^q \, w ,
\end{align*}
since we may let $L\to\infty$ to obtain the claim.
Denote
\[
E = B(0,L) \cap \{ \lambda < M^{\gamma+}_\xi f \leq 2\lambda \} .
\]
By Lemma~\ref{lem:trunctation},
it is enough to show the claim for truncations $\max\{w,a\}$, $a>0$.
Hence, assume that $w$ is bounded from below by $a$.

By the definition of $E$,
for every $z \in E$ there exists a parabolic rectangle $R_{z}$ such that 
$z$ is the center point of $R^-_{z}(\gamma)$, $l(R_{z})\geq \xi$ and
\[
\lambda < \dashint_{R^+_{z}(\gamma)} \lvert f \rvert \leq 2 \lambda .
\]
Since $f\in L^1(\mathbb{R}^{n+1})$, we have
\[
\lvert R^+_{z}(\gamma) \rvert < \frac{1}{\lambda} \int_{R^+_{z}(\gamma)} \lvert f \rvert
\leq \frac{1}{\lambda} \int_{\mathbb{R}^{n+1}} \lvert f \rvert
< \infty,
\]
and thus the side length of $R_{z}$ is bounded from above.
Denote $P_z=5R_z$ and consider $P_z^\pm(\alpha)$ with $\alpha = \gamma/5^p$. We observe that the top time coordinates of $P_z^-(\alpha)$ and $R_z^-(\gamma)$ coincide and similarly the bottom time coordinates of $P_z^+(\alpha)$ and $R_z^+(\gamma)$.
In particular, we have $R_z^\pm(\gamma)\subset P_z^\pm(\alpha)$.
Since $\bigcup_{z\in E} P_{z}$ is a bounded set, the absolute continuity of the integral implies that
there is $0<\varepsilon<1$ such that
\begin{align*}
w((1+\varepsilon)P_z^-(\alpha) \setminus P_z^-(\alpha)) \leq a (1-\alpha)\xi^{n+p} \leq a \lvert P_z^-(\alpha) \rvert \leq w(P_z^-(\alpha))
\end{align*}
for every $z\in E$.
This also implies that
\[
w((1+\varepsilon)P_z^-(\alpha)) \leq 2 w(P_z^-(\alpha)) .
\]
Since $E$ is compact, there exists a finite collection of balls $B(z_k, (1-\alpha) \xi^p \varepsilon /2)$ with $z_k = (x_k,t_k) \in E$ such that
\[
E \subset \bigcup_{k} B(z_k, (1-\alpha) \xi^p \varepsilon /2) .
\]

We choose $R_{z_k}$ which top has a largest time coordinate and denote it by $R_{z_{k_1}}$. 
Then we consider $R_{z_k}$ which top has a second largest time coordinate.
If $z_k$ is contained in $P^-_{z_{k_1}}(\alpha)$, we discard $R_{z_k}$. Otherwise, we select $R_{z_k}$ and denote it by $R_{z_{k_2}}$.
Suppose that $R_{z_{k_l}}$, $l=1,\dots,m-1$, have been chosen.
Consider $R_{z_k}$ with a next largest top time coordinate and select it if $z_k$ is not contained in any previously chosen $P^-_{z_{k_l}}(\alpha)$ and denote it by $R_{z_{k_m}}$.
In this manner
we obtain a finite collection $\{ R_{z_{k_l}} \}_l$.
We proceed to the second selection. Denote $R_{z_{k_l}}$ with a largest side length by $R_1$.
Suppose that $R_j$, $j=1,\dots,i-1$, have been chosen.
Then consider $R_{z_{k_l}}$ with a next largest side length and select it if 
$P^-_{z_{k_l}}(\alpha) \nsubseteq P_j^-(\alpha) $ for every $j=1,\dots, i-1$, and denote it by $R_i$.
Hence, we obtain a finite collection $\{R_{i}\}_i$.
Note that $R_i^-(\gamma) \nsubseteq R_j^-(\gamma) $ for $i\neq j$.
It follows that
every $z_k$ is contained in some $P^-_{i}(\alpha)$ and we have
$ B(z_k, (1-\alpha) \xi^p \varepsilon /2) \subset (1+\varepsilon)P^-_{i}(\alpha) $.
This implies that
\[
E \subset \bigcup_{k} B(z_k, (1-\alpha) \xi^p \varepsilon /2)
\subset \bigcup_{i} (1+\varepsilon) P^-_{i}(\alpha) .
\]
Therefore, we obtain
\begin{equation}
\label{eq:superlevelset}
w(E) \leq \sum_i w((1+\varepsilon) P^-_{i}(\alpha)) \leq 2 \sum_i w(P_i^-(\alpha)) .
\end{equation}
By the construction, we observe that for given $k\in\mathbb{Z}$ the corresponding $R_i^-(\gamma)$ of side length $2^{-k-1} < l(R_i) \leq 2^{-k}$ are pairwise disjoint.

Fix $i$ and denote 
\[
J_i = \{j\in\mathbb{N}: R_i^+(\gamma) \cap R_j^+(\gamma) \neq \emptyset, l(R_j) < l(R_i) \} .
\]
Divide $J_i$ into two subcollections
\[
J_i^1 = \{j\in J_i : R_j^+(\gamma) \nsubseteq R_i^+(\gamma) \} 
\quad\text{and}\quad
J_i^2 = \{j\in J_i : R_j^+(\gamma) \subset R_i^+(\gamma) \}.
\]
We first consider $J_i^1$.
Let $2^{-k_0-1}<l(R_i)\leq 2^{-k_0}$ and $2^{-k-1}<l(R_j)\leq 2^{-k}$.
If $R_j^+(\gamma) \nsubseteq R_i^+(\gamma)$, then
$R_j^+(\gamma)$ intersects the boundary of $R_i^+(\gamma)$ and we have
$R_j \subset A_k$ where $A_k$ 
is a set around the boundary of $R_i^+(\gamma)$ such that
\begin{align*}
\lvert A_k \rvert &\leq 
2\bigl(l(R_i) + 2l(R_i)\bigr)^n 2^{-kp+2}
+ 2n \bigl( l(R_i) + 2l(R_i) \bigr)^{n-1} \\
&\qquad \cdot
\bigl( (1-\gamma) l(R_i)^p + 4 l(R_i)^{p}\bigr)
2^{-k+1} 
\\
&\leq 2^{2n+3} \bigl( l(R_i)^n 2^{-kp} + n l(R_i)^{n-1+p} 2^{-k} \bigr)
.
\end{align*}
Since $R_j^-(\gamma)$ of approximately same side length are pairwise disjoint,
we obtain
\begin{align*}
\sum_{j\in J_i^1} \lvert R_j^+(\gamma) \lvert 
&= \sum_{k=k_0}^\infty \sum_{\substack{j\in J_i^1 \\ 2^{-k-1}<l(R_j)\leq 2^{-k} }} \lvert R_j^+(\gamma) \lvert
= \sum_{k=k_0}^\infty \sum_{\substack{j\in J_i^1 \\ 2^{-k-1}<l(R_j)\leq 2^{-k} }} \lvert R_j^-(\gamma) \lvert \\
&\leq 
\sum_{k=k_0}^\infty \lvert A_k \rvert
\leq 2^{2n+3} \biggl( l(R_i)^n \sum_{k=k_0}^\infty 2^{-kp} + n l(R_i)^{n-1+p} \sum_{k=k_0}^\infty 2^{-k} \biggr) \\
&\leq 2^{2n+3} \biggl( l(R_i)^n 2^{-k_0 p+1} + n l(R_i)^{n-1+p} 2^{-k_0+1} \biggr) \\
&\leq 2^{2n+4} \biggl( l(R_i)^n 2^p l(R_i)^p + n l(R_i)^{n-1+p} 2 l(R_i) \biggr) \\
&\leq
\frac{C_1}{1-\gamma} \lvert R_i^+(\gamma) \rvert ,
\end{align*}
where $C_1 = 2^{3n+p+4}$.
We proceed to estimate $J_i^2$.
Let 
\[
\mathcal{F}_{k_0} = \{ j\in J_i^2: 2^{-k_0-1}<l(R_j)\leq 2^{-k_0} \}
\]
and
\[
\mathcal{F}_{k} = \Bigl\{ j\in J_i^2: 2^{-k-1}<l(R_j)\leq 2^{-k}, R_j^-(\gamma) \cap \bigcup_{l=k_0}^{k-1} \bigcup_{m \in \mathcal{F}_l } R_m^-(\gamma) = \emptyset \Bigr\},
\quad k>k_0.
\]
In particular, we select those rectangles to $\mathcal{F}_k$ that 
are approximately of side length $2^{-k}$ and do not intersect the previously chosen rectangles.
Define $\mathcal{F} = \bigcup_{k=k_0}^\infty \mathcal{F}_k$.
For every $m\in J_i^2 \setminus \mathcal{F} $, there exists $j\in \mathcal{F}$ such that $R_j^-(\gamma) \cap R_m^-(\gamma) \neq \emptyset$ and $l(R_m)< l(R_j)$.
Moreover, recall that $R_m^-(\gamma) \nsubseteq R_j^-(\gamma) $ by the selection process.
Thus, we have
\begin{align*}
\sum_{j\in J_i^2} \lvert R_j^+(\gamma) \rvert \leq \sum_{j\in \mathcal{F}} \biggl( \lvert R_j^+(\gamma) \rvert +  
\sum_{m\in\mathcal{G}_j} \lvert R_m^+(\gamma) \rvert
\biggr) ,
\end{align*}
where
\[
\mathcal{G}_j = \{ m \in J_i^2 \setminus \mathcal{F}: R_j^-(\gamma) \cap R_m^-(\gamma) \neq \emptyset, R_m^-(\gamma) \nsubseteq R_j^-(\gamma) , l(R_m) < l(R_j) \} .
\]
Then a similar argument as for $J_i^1$ can be applied for the second sum above.
More precisely,
let $2^{-k_0-1}<l(R_j)\leq 2^{-k_0}$ and $2^{-k-1}<l(R_m)\leq 2^{-k}$.
If $R_m^-(\gamma) \nsubseteq R_j^-(\gamma)$, then
$R_m^-(\gamma)$ intersects the boundary of $R_j^-(\gamma)$ and we have
$R_m^-(\gamma) \subset A_k$ where $A_k$ 
is a set around the boundary of $R_j^-(\gamma)$ such that
\begin{align*}
\lvert A_k \rvert 
&\leq
2\bigl( l(R_j) + 2l(R_j) \bigr)^n (1-\gamma) 2^{-kp+1}
+2n \bigl( l(R_j) + 2l(R_j) \bigr)^{n-1} \\
&\qquad \cdot
\bigl( (1-\gamma) l(R_j)^{p} + 2(1-\gamma) l(R_j)^{p} \bigr)
2^{-k+1} \\
&\leq 
(1-\gamma) 2^{2n+2} \bigl( l(R_j)^n 2^{-kp} + n l(R_j)^{n-1+p} 2^{-k} \bigr) .
\end{align*}
We have
\begin{align*}
\sum_{m\in \mathcal{G}_j} \lvert R_m^+(\gamma) \lvert 
&= \sum_{k=k_0}^\infty \sum_{\substack{m\in \mathcal{G}_j \\ 2^{-k-1}<l(R_m)\leq 2^{-k} }} \lvert R_m^+(\gamma) \lvert
= \sum_{k=k_0}^\infty \sum_{\substack{m\in \mathcal{G}_j \\ 2^{-k-1}<l(R_m)\leq 2^{-k} }} \lvert R_m^-(\gamma) \lvert \\
&\leq 
\sum_{k=k_0}^\infty \lvert A_k \rvert
\leq 
2^{2n+2} (1-\gamma) \biggl( l(R_j)^n \sum_{k=k_0}^\infty 2^{-kp} + n l(R_j)^{n-1+p} \sum_{k=k_0}^\infty 2^{-k} 
\biggr) \\
&\leq 2^{2n+2} (1-\gamma) \biggl( l(R_j)^n 2^{-k_0 p +1} + n l(R_j)^{n-1+p} 2^{-k_0+1} \biggr) \\
&\leq 2^{2n+3} (1-\gamma) \biggl( l(R_j)^n 2^p l(R_j)^p + n l(R_j)^{n-1+p} 2 l(R_j) \biggr) \\
&\leq
\frac{C_1}{2} \lvert R_j^+(\gamma) \rvert .
\end{align*}
Observe that $R_j^+(\gamma) \subset R_i^+(\gamma)$ implies $R_j^-(\gamma) \subset R_i$.
By the previous estimate, the definition of $\mathcal{F}$ and $\mathcal{F}_k$
and the pairwise disjointness of $R_j^-(\gamma)$ in $\mathcal{F}_k$, 
we obtain
\begin{align*}
\sum_{j\in J_i^2} \lvert R_j^+(\gamma) \rvert &\leq \sum_{j\in \mathcal{F}} \biggl( \lvert R_j^+(\gamma) \rvert +  
\sum_{m \in \mathcal{G}_j} \lvert R_m^+(\gamma) \rvert
\biggr)
\leq
\sum_{j\in \mathcal{F}} \biggl( \lvert R_j^+(\gamma) \rvert +  
\frac{C_1}{2} \lvert R_j^+(\gamma) \rvert
\biggr) \\
&\leq C_1 \sum_{k=k_0}^\infty \sum_{j\in \mathcal{F}_k} \lvert R_j^-(\gamma) \rvert 
\leq C_1 \sum_{k=k_0}^\infty 
\Bigl\lvert \bigcup_{j\in \mathcal{F}_k} R_{j}^-(\gamma) \Bigr\rvert 
\\
&= C_1 \Bigl\lvert \bigcup_{k=k_0}^\infty \bigcup_{j\in \mathcal{F}_k} R_{j}^-(\gamma) \Bigr\rvert 
\leq
C_1 \Bigl\lvert \bigcup_{j\in J_i^2} R_{j}^-(\gamma) \Bigr\rvert
\\
&\leq C_1 \lvert R_i \rvert 
= \frac{2C_1}{1-\gamma} \lvert R_i^+(\gamma) \rvert .
\end{align*}
Combining the estimates for $J_i^1$ and $J_i^2$, we get
\begin{align*}
\sum_{j\in J_i} \lvert R_j^+(\gamma) \rvert &= \sum_{j\in J_i^1} \lvert R_j^+(\gamma) \rvert + \sum_{j\in J_i^2} \lvert R_j^+(\gamma) \rvert 
\leq \frac{C_1}{1-\gamma} \lvert R_i^+(\gamma) \rvert + \frac{2C_1}{1-\gamma} \lvert R_i^+(\gamma) \rvert \\
&\leq \frac{2^{3n+p+6} }{1-\gamma} \lvert R_i^+(\gamma) \rvert
\leq \frac{c}{2} \lvert R_i^+(\gamma) \rvert ,
\end{align*}
where $c = \lceil 2^{3n+p+7} / (1-\gamma) \rceil$.
Then we have
\begin{align*}
\sum_{j\in J_i} \int_{R_j^+(\gamma)} \lvert f \rvert \leq 2\lambda \sum_{j\in J_i} \lvert R_j^+(\gamma) \rvert \leq c \lambda \lvert R_i^+(\gamma) \rvert
\leq c \int_{R_i^+(\gamma)} \lvert f \rvert .
\end{align*}

We extract from $\{R_i^+(\gamma)\}_i$ a collection of subsets $F_i$ that have bounded overlap.
Fix $i$ and denote the number of indices in $J_i$ by $N$.
If $N\leq 2c$, we choose $F_i = R_i^+(\gamma)$.
Otherwise, if $N>2c$, we define
\[
G_i^k = \{z\in R_i^+(\gamma): \sum_{j\in J_i} \chi_{R_j^+(\gamma)}(z) \geq k \}, \quad k\in\mathbb{N}.
\]
For $k_1 < k_2$, we have $G_i^{k_2} \subset G_i^{k_1} $.
Moreover, observe that 
\begin{align*}
\sum_{k=1}^N \chi_{G_i^k}(z) = \sum_{j\in J_i} \chi_{R_j^+(\gamma)}(z) .
\end{align*}
Then it follows that
\begin{align*}
2 c \int_{G_i^{2c}} \lvert f \rvert &\leq N \int_{G_i^{2c}} \lvert f \rvert \leq \sum_{k=1}^N \int_{G_i^k} \lvert f \rvert = \int_{R_i^+(\gamma)} \lvert f \rvert \sum_{k=1}^N \chi_{G_i^k} \\
&= \int_{R_i^+(\gamma)} \lvert f \rvert \sum_{j\in J_i} \chi_{R_j^+(\gamma)}
\leq \sum_{j\in J_i} \int_{R_j^+(\gamma)} \lvert f \rvert
\leq c \int_{R_i^+(\gamma)} \lvert f \rvert .
\end{align*}
Define $F_i = R_i^+(\gamma) \setminus G_i^{2c}$ for which we have
\begin{align*}
\int_{F_i} \lvert f \rvert = \int_{R_i^+(\gamma)} \lvert f \rvert - \int_{G_i^{2c}} \lvert f \rvert \geq \frac{1}{2} \int_{R_i^+(\gamma)} \lvert f \rvert > \frac{\lambda}{2} \lvert R_i^+(\gamma) \rvert .
\end{align*}
Thus, for every $F_i$ it holds that
\begin{equation}
\label{eq:intavgFi}
\frac{2}{\lvert R_i^+(\gamma) \rvert} \int_{F_i} \lvert f \rvert > \lambda .
\end{equation}
Moreover, for every $z\in F_{i}$ we have 
\begin{equation}
\label{eq:bddoverlapFi}
\sum_{j\in J_{i}} \chi_{R_j^+(\gamma)}(z) \leq 2c .
\end{equation}

Next we observe that $R_i^+(\gamma)$ of approximately same side length have bounded overlap. 
Fix $z\in\mathbb{R}^{n+1}$ and denote
\[
I_z^k = \{i\in\mathbb{N}: 2^{-k-1} < l(R_i) \leq 2^{-k}, z\in R_i^+(\gamma) \}, \quad k\in\mathbb{Z}.
\]
Then we have
\[
\bigcup_{i \in I_z^k} 
R_i^-(\gamma)
\subset 
S_z
,
\]
where 
$S_z$ is a rectangle centered at $z$ with side lengths $l_x(S_z) = 2^{-k+1}$ and $l_t(S_z) = 2^{-kp+1} $.
Since $R_i^-(\gamma)$ are pairwise disjoint,
we obtain 
\begin{align*}
\sum_{\substack{ i \\ 2^{-k-1} < l(R_i) \leq 2^{-k}}} 
\chi_{R_i^+(\gamma)}(z) 
&= 
\sum_{i\in I_z^k}
\frac{1}{(1-\gamma) l(R_i)^{n+p} } \lvert R_i^-(\gamma) \rvert \\
&\leq
\frac{2^{(k+1)(n+p)}}{1-\gamma } 
\Bigl\lvert 
\bigcup_{i\in I_z^k}
R_i^-(\gamma) \Bigr\rvert 
\leq\frac{2^{(k+1)(n+p)}}{1-\gamma } 
\lvert S_z \rvert \\
&=\frac{2^{(k+1)(n+p)}}{1-\gamma } 
2^{(-k+1)n} 2^{-kp+1}
= \frac{2^{2n+p+1}}{1-\gamma}  .
\end{align*}
We show that $F_i$ have a bounded overlap.
Fix $z\in \bigcup_i F_i$.
Consider a rectangle $R_{i_0} \in \{R_i\}_i$ with a largest side length such that $z\in F_{i_0}$.
Let $2^{-k_0-1}< l(R_{i_0}) \leq 2^{-k_0}$.
By the previous estimate and~\eqref{eq:bddoverlapFi}, we have
\begin{align*}
\sum_i \chi_{F_i}(z) 
&= \sum_{\substack{i \\ 2^{-k_0-1}< l(R_i) \leq 2^{-k_0} }} \chi_{F_i}(z) 
+ \sum_{\substack{i \\ l(R_i)\leq 2^{-k_0-1} }} \chi_{F_i}(z) \\
&\leq \sum_{\substack{i \\ 2^{-k_0-1}< l(R_i)\leq 2^{-k_0} }} \chi_{R_i^+(\gamma)}(z)
+ \sum_{i \in J_{i_0}} \chi_{R_i^+(\gamma)}(z) \\
&\leq 
\frac{2^{2n+p+1}}{1-\gamma} + 2c \leq C_2 ,
\end{align*}
where $C_2 = 2^{3n+p+9} /(1-\gamma) $.

We first consider the case $1<q<\infty$.
By~\eqref{eq:superlevelset},~\eqref{eq:intavgFi}, H\"older's inequality,
Theorem~\ref{thm:timelagchange} with the constant $C$
and the bounded overlap of $F_i$, we conclude that
\begin{align*}
w(E) &\leq 
2 \sum_i w(P_i^-(\alpha))
\leq \frac{2^{q+1}}{\lambda^q} \sum_i w(P_i^-(\alpha)) \biggl( \frac{1}{\lvert R_i^+(\gamma) \rvert} \int_{F_i} \lvert f \rvert \biggr)^q \\
&\leq \frac{2^{q+1}}{\lambda^q} \sum_i  \frac{w(P_i^-(\alpha))}{\lvert R_i^+(\gamma) \rvert^q} \biggl( \int_{F_i} w^{-\frac{1}{q-1}} \biggr)^{q-1} \int_{F_i} \lvert f \rvert^q \, w \\
&\leq \frac{2^{q+1}}{\lambda^q} 
\biggl( \frac{1-\alpha}{1-\gamma} 5^{n+p} \biggr)^q
\sum_i \dashint_{P_i^-(\alpha)} w  \biggl( \dashint_{P_i^+(\alpha)} w^{-\frac{1}{q-1}} \biggr)^{q-1} \int_{F_i} \lvert f \rvert^q \, w \\
&\leq 
\frac{C_3 C}{\lambda^q} 
\sum_i \int_{F_i} \lvert f \rvert^q \, w
\leq
\frac{C_2 C_3 C}{\lambda^q} 
\int_{\mathbb{R}^{n+1}} \lvert f \rvert^q \, w ,
\end{align*}
where $C_3 = 2^{q+1} 5^{q(n+p)} / (1-\gamma)^q$.
For the case $q=1$, we use~\eqref{eq:superlevelset},~\eqref{eq:intavgFi}, 
Theorem~\ref{thm:timelagchange} with the constant $C$
and the bounded overlap of $F_i$ to get
\begin{align*}
w(E) &\leq 
2 \sum_i w(P_i^-(\alpha))
\leq \frac{4}{\lambda} \sum_i \frac{w(P_i^-(\alpha))}{\lvert R_i^+(\gamma) \rvert} \int_{F_i} \lvert f \rvert \\
&= 
\frac{4}{\lambda} 
\frac{1-\alpha}{1-\gamma} 5^{n+p}
\sum_i \int_{F_i} \lvert f \rvert \, w_{P_i^-(\alpha)} \\
&\leq
\frac{C_3 C}{\lambda} 
\sum_i \int_{F_i} \lvert f \rvert \, w 
\leq
\frac{C_2 C_3 C}{\lambda} 
\int_{\mathbb{R}^{n+1}} \lvert f \rvert \, w .
\end{align*}
Observe that in the case $\gamma=0$, we do not need to use Theorem~\ref{thm:timelagchange}.
This completes the proof.
\end{proof}

Using the self-improving property of the parabolic Muckenhoupt weights, we obtain the strong type estimate for the uncentered parabolic maximal function.

\begin{theorem}
\label{thm:strongtype}
Let $1<q<\infty$, $0<\gamma<1$ and $w$ be a weight.
Then $w \in A^+_{q}(\gamma)$ if and only if
there exists a constant $C=C(n,p,q,\gamma,[w]_{A^+_q(\gamma)})$ such that
\[
\int_{\mathbb{R}^{n+1}} (M^{\gamma+}f)^q \, w \leq C \int_{\mathbb{R}^{n+1}} \lvert f \rvert^q \, w 
\]
for every $f\in L^1_{\mathrm{loc}}(\mathbb{R}^{n+1})$.
\end{theorem}

\begin{proof}
Assume that the strong type estimate holds.
Since it implies the weak type estimate in Theorem~\ref{thm:weaktype}, we have $w \in A^+_{q}(\gamma)$.

We show the reverse implication.
Corollary~\ref{cor:Aqselfimprove} and Theorem~\ref{thm:weaktype} 
imply that there exist $\varepsilon>0$ and $C>0$ such that
\[
w( \{ M^{\gamma+} f > \lambda \}) \leq
\frac{C}{\lambda^{q-\varepsilon}} \int_{\mathbb{R}^{n+1}} \lvert f \rvert^{q-\varepsilon} \, w .
\]
Moreover, observe that
\[
\norm{M^{\gamma+} f}_{L^\infty(w)} \leq \norm{f}_{L^\infty(w)} .
\]
In other words, $M^{\gamma+}f$ is bounded from $L^{q-\varepsilon}(w)$ to $L^{q-\varepsilon,\infty}(w)$ and from $L^\infty(w)$ to $L^{\infty}(w)$.
The Marcinkiewicz interpolation theorem implies that $M^{\gamma+}f$ is bounded from $L^q(w)$ to $L^{q}(w)$, that is
\begin{align*}
\int_{\mathbb{R}^{n+1}} (M^{\gamma+}f)^q \, w &\leq
 \frac{ q}{q-(q-\varepsilon)}  2^q  C
\int_{\mathbb{R}^{n+1}} \lvert f \rvert^q \, w
=
\frac{q 2^q  C}{\varepsilon} 
\int_{\mathbb{R}^{n+1}} \lvert f \rvert^q \, w .
\end{align*}
\end{proof}

\section{Factorization and characterization of parabolic Muckenhoupt weights}
\label{sec:A1theory}

The following theorem is a
Jones factorization for parabolic Muckenhoupt weights. 
This extends \cite[Theorem~6.3]{kinnunenSaariParabolicWeighted} to a full range of the time lag.

\begin{theorem}
\label{thm:jones}
Let $0<\gamma<1$. A weight $w$ is in $A_q^+(\gamma)$ if and only if $w=uv^{1-q}$, where $u\in A_1^+(\gamma)$ and $v\in A_1^-(\gamma)$.
\end{theorem}

\begin{proof}
Assume first that $u\in A_1^+(\gamma)$ and $v\in A_1^-(\gamma)$. Let $R\subset\mathbb{R}^{n+1}$ be a parabolic rectangle.
By the definition of $A_1^+(\gamma)$ and $A_1^-(\gamma)$, we have
\[
u(x,t)^{-1} 
\leq \biggl( \essinf_{(y,s)\in R^+(\gamma)} u(y,s) \biggr)^{-1}
\leq [u]_{A^+_1(\gamma)} \biggl( \dashint_{R^-(\gamma)} u \biggr)^{-1}
\]
for almost every $(x,t)\in R^+(\gamma)$, and
\[
v(x,t)^{-1}
\leq \biggl( \essinf_{(y,s)\in R^-(\gamma)} v(y,s) \biggr)^{-1}
\leq [v]_{A^-_1(\gamma)} \biggl( \dashint_{R^+(\gamma)} v \biggr)^{-1}
\]
for almost every $(x,t)\in R^-(\gamma)$.
By using these estimates, we obtain
\begin{align*}
&\biggl( \dashint_{R^-(\gamma)} uv^{1-q} \biggr) \biggl( \dashint_{R^+(\gamma)} (uv^{1-q})^{\frac{1}{1-q}} \biggr)^{q-1}
= \biggl( \dashint_{R^-(\gamma)} uv^{1-q} \biggr) \biggl( \dashint_{R^+(\gamma)} u^{\frac{1}{1-q}} v \biggr)^{q-1} \\
&\qquad\leq  [v]_{A^-_1(\gamma)} \biggl( \dashint_{R^+(\gamma)} v \biggr)^{1-q} \biggl( \dashint_{R^-(\gamma)} u \biggr)  [u]_{A^+_1(\gamma)} \biggl( \dashint_{R^-(\gamma)} u \biggr)^{-1}
\biggl( \dashint_{R^+(\gamma)} v \biggr)^{q-1} \\
&\qquad= [u]_{A^+_1(\gamma)} [v]_{A^-_1(\gamma)} .
\end{align*}
This shows that $uv^{1-q} \in A_q^+(\gamma)$.

Next we prove the other direction. Let $w\in A_q^+(\gamma)$ with $q\geq2$. Define the operator $T$ 
for nonnegative functions $f$
as follows
\[
Tf = \big(w^{-\frac{1}{q}} M^{\gamma-} ( w^\frac{1}{q}f^{q-1} ) \big)^\frac{1}{q-1} + w^\frac{1}{q} M^{\gamma+} (w^{-\frac{1}{q}} f) .
\]
If $f\in L^{q}$, then $w^\frac{1}{q}f^{q-1} \in L^{q'}(w^{1-q'})$ and $w^{-\frac{1}{q}} f \in L^q(w) $.
Note that $w^{1-q'}  \in A_{q'}^-(\gamma)$
by Lemma~\ref{lem:dualityAp}.
The operator $T$ is bounded on $L^q$,
since the boundedness of $M^{\gamma-}$ on $L^{q'}(w^{1-q'})$ and $M^{\gamma+}$ on $L^q(w)$ (see Theorem~\ref{thm:strongtype})
imply
\begin{align*}
\int_{\mathbb{R}^{n+1}} \bigl(w^{-\frac{1}{q}} M^{\gamma-} ( w^\frac{1}{q}f^{q-1} ) \bigr)^\frac{q}{q-1} &= \int_{\mathbb{R}^{n+1}} \big( M^{\gamma-} ( w^\frac{1}{q}f^{q-1} ) \big)^{q'} w^{1-q'} \\
&\leq c_1 \int_{\mathbb{R}^{n+1}} \lvert w^\frac{1}{q}f^{q-1} \rvert^{q'} w^{1-q'} = c_1 \int_{\mathbb{R}^{n+1}} \lvert f \rvert^{q}
\end{align*}
and
\begin{align*}
\int_{\mathbb{R}^{n+1}} \bigl( w^\frac{1}{q} M^{\gamma+} (w^{-\frac{1}{q}} f) \bigr)^q 
&= \int_{\mathbb{R}^{n+1}} \bigl( M^{\gamma+} (w^{-\frac{1}{q}} f) \bigr)^q w\\
& \leq c_2 \int_{\mathbb{R}^{n+1}} \lvert w^{-\frac{1}{q}} f \rvert^q w = c_2 \int_{\mathbb{R}^{n+1}} \lvert f \rvert^q .
\end{align*}
Moreover, since $q\geq2$, we observe that $T$ is subadditive.
Thus, we have
\[
\norm{Tf}_{L^q} \leq c \norm{f}_{L^q} \quad\text{and}\quad T(f+g)\leq Tf + Tg
\]
for every nonnegative $f,g \in L^q$.

Fix a nonnegative $f \in L^q$ and set
\[
\eta = \sum_{k=1}^\infty (2c)^{-k} T^k f ,
\]
where $T^k = T^{k-1} \circ T$ is the iterated operator.
Since
\[
\norm{\eta}_{L^q} \leq \sum_{k=1}^\infty (2c)^{-k} \lVert T^k f \rVert_{L^q} \leq \sum_{k=1}^\infty (2c)^{-k} c^k \norm{f}_{L^q} = \norm{f}_{L^q} ,
\]
the series in the definition of $\eta$ converges absolutely and by the completeness of $L^q$ we have $\eta\in L^q$.
By the properties of $T$, we get
\begin{align*}
T\eta \leq \sum_{k=1}^\infty (2c)^{-k} T^{k+1} f = \sum_{k=2}^\infty (2c)^{1-k} T^k f \leq 2c\eta .
\end{align*}
By defining $u = w^\frac{1}{q} \eta^{q-1}$ and $v = w^{-\frac{1}{q}} \eta$,
it holds that $w = uv^{1-q}$,
\begin{align*}
M^{\gamma-} u = M^{\gamma-}( w^\frac{1}{q} \eta^{q-1} ) \leq w^\frac{1}{q} (T\eta)^{q-1} \leq w^\frac{1}{q} (2c\eta)^{q-1} = (2c)^{q-1} u
\end{align*}
and
\begin{align*}
M^{\gamma+} v = M^{\gamma+} (w^{-\frac{1}{q}} \eta) \leq w^{-\frac{1}{q}} T\eta \leq w^{-\frac{1}{q}} 2c\eta = 2c v .
\end{align*}
Hence, by Proposition~\ref{thm:maximalA1-cond}, we conclude that $u\in A_1^+(\gamma)$ and $v\in A_1^-(\gamma)$.

It is left to consider the case $1<q<2$. 
By Lemma~\ref{lem:dualityAp},
we know that $w^{1-q'}\in A_{q'}^-(\gamma)$.
Thus, we may apply the argument above to obtain $u\in A_1^-(\gamma)$ and $v\in A_1^+(\gamma)$ such that 
$w^{1-q'} = uv^{1-q'}$,
which is equivalent with $w = v u^{1-q}$.
This finishes the proof.
\end{proof}

The next theorem is a Coifman--Rochberg characterization for parabolic Muckenhoupt weights. 
This extends \cite[Lemma~6.4]{kinnunenSaariParabolicWeighted} to a full range of the time lag with the uncentered parabolic maximal function.
Part $(i)$ also holds for $\gamma=0$ and $p=1$.

\begin{theorem}
\label{thm:coifmanrochberg}
Let $0<\gamma<1$.
\begin{enumerate}[(i)]
\item Assume that $f\in L^1_{\mathrm{loc}}(\mathbb{R}^{n+1})$ and $M^{\gamma-}f<\infty$ almost everywhere. Let $0<\delta<1$. Then $(M^{\gamma-}f)^\delta$ is an $A_1^+(\gamma)$ weight with 
$[w]_{A_1^+(\gamma)}$
depending only on $n$ and $\delta$.
\item Assume that $w\in A_1^+(\gamma)$. Then there exist $f\in L^1_{\mathrm{loc}}(\mathbb{R}^{n+1})$, $0<\delta<1$ and $b$ with $b,\frac{1}{b}\in L^\infty$ such that $w=b (M^{\gamma-}f)^\delta$ almost everywhere.
\end{enumerate}
\end{theorem}

\begin{proof}

Let $R\subset\mathbb{R}^{n+1}$ be a parabolic rectangle with side length $L$. 
Consider a parabolic rectangle $P$ with $l(P)=3L$  
such that the spatial centers and
the top time coordinates of $P^-(\gamma)$ and $R^-(\gamma)$ coincide.
Note that $R^-(\gamma) \subset P^-(\gamma)$.
By choosing $\tau\geq1$ such that
\begin{align*}
\tau (1+\gamma) L^p - (1-\gamma)L^p = 2\gamma(3L)^p ,
\quad\text{that is}\quad
\tau = \frac{1+(2\cdot 3^{p}-1)\gamma}{1+\gamma} ,
\end{align*}
we have
$S^+(\gamma) = R^-(\gamma) + (0, \tau (1+\gamma) L^p) \subset P^+(\gamma)$.
We decompose 
$f=f_1+f_2$ where
\[
f_1(z) = f(z) \chi_{P^-(\gamma)}(z)
\quad\text{and}\quad
f_2(z) = f(z) \chi_{\mathbb{R}^{n+1} \setminus P^-(\gamma)}(z) .
\]
By the weak type $(1,1)$ estimate for the parabolic maximal function there exists a constant $C_1$ such that
\begin{align*}
\dashint_{R^-(\gamma)} (M^{\gamma-}f_1)^\delta
&= \frac{\delta}{\lvert R^-(\gamma) \rvert} \int_0^\infty \lambda^{\delta-1} \lvert R^-(\gamma) \cap \{ M^{\gamma-}f_1> \lambda\} \rvert \dla \\
&=
\frac{\delta}{\lvert R^-(\gamma) \rvert} \int_0^{\norm{f_1}_{L^1}/\lvert R^-(\gamma) \rvert} \lambda^{\delta-1} \lvert R^-(\gamma) \cap \{ M^{\gamma-}f_1> \lambda\} \rvert \dla \\
&\qquad+ \frac{\delta}{\lvert R^-(\gamma) \rvert} \int_{\norm{f_1}_{L^1}/\lvert R^-(\gamma) \rvert}^\infty \lambda^{\delta-1} \lvert R^-(\gamma) \cap \{ M^{\gamma-}f_1> \lambda\} \rvert \dla \\
&\leq 
\delta \int_0^{\norm{f_1}_{L^1}/\lvert R^-(\gamma) \rvert} \lambda^{\delta-1} \dla 
+ \frac{\delta}{\lvert R^-(\gamma) \rvert} \int_{\norm{f_1}_{L^1}/\lvert R^-(\gamma) \rvert}^\infty \lambda^{\delta-2} C_1 \norm{f_1}_{L^1} \dla \\
&=
\biggl(\frac{\norm{f_1}_{L^1}}{\lvert R^-(\gamma) \rvert}\biggr)^\delta
+ \frac{C_1 \delta}{\delta-1} \frac{\norm{f_1}_{L^1} }{\lvert R^-(\gamma) \rvert} 
\biggl(\frac{\norm{f_1}_{L^1}}{\lvert R^-(\gamma) \rvert}\biggr)^{\delta-1} \\
&= 
\biggl(1+ \frac{C_1 \delta}{\delta-1} \biggr)
\biggl( \frac{1}{\lvert R^-(\gamma) \rvert} \int_{P^-(\gamma)} \lvert f \rvert  \biggr)^\delta \\
&= 
\biggl(1+ \frac{C_1 \delta}{\delta-1} \biggr)
3^{\delta(n+p)}  
\biggl( \dashint_{P^-(\gamma)} \lvert f \rvert  \biggr)^\delta 
\leq 
C_2
(M^{\gamma-}f(z))^\delta
\end{align*}
for every $z\in S^+(\gamma) \subset P^+(\gamma)$, where 
$C_2 =  (1 + \frac{C_1 \delta}{1-\delta}) 3^{\delta(n+p)} $.
Thus, it holds that
\[
\dashint_{R^-(\gamma)} (M^{\gamma-}f_1)^\delta \leq
C_2
\essinf_{S^+(\gamma) \ni z} (M^{\gamma-}f(z))^\delta .
\]

Let $y\in R^-(\gamma)$ and $U$ be an arbitrary parabolic rectangle such that $y\in U^+(\gamma)$.
If $U^-(\gamma)\subset P^-(\gamma)$, then $ \lvert f_2 \vert_{U^-(\gamma)} = 0$.
Thus, we may assume that $U^-(\gamma)$ intersects the complement of $P^-(\gamma)$, which implies that
\[
l(U) > \min\Bigl\{ L, \frac{ (3^p-1)^\frac{1}{p} }{ 2^\frac{1}{p} } (1-\gamma)^\frac{1}{p} L  \Bigr\}
\geq  
(1-\gamma)^\frac{1}{p} L .
\]
We enlarge $U$ so that the upper part of the enlarged parabolic rectangle contains $S^+(\gamma)$.
We choose $\sigma\geq1$ in the following way.
Let $V$ be the smallest enlarged rectangle such that $l(V)=\sigma l(U)$, the spatial centers of $U$ and $V$ coincide, the top time coordinates of $V^-(\gamma)$ and $U^-(\gamma)$ coincide and $S^+(\gamma)\subset V^+(\gamma)$.
We compute an upper bound for $\sigma$ which tells how much $U$ needs to be enlarged for $S^+(\gamma)\subset V^+(\gamma)$ to hold.
A rough upper bound can be obtained from
\begin{align*}
\tau(1+\gamma)L^p
+ (1-\gamma) (3L)^p
\geq (1+\gamma)(1-\gamma) (\sigma L)^p ,
\end{align*}
that is
\begin{align*}
\sigma^p &\leq 
\frac{\tau(1+\gamma) + 3^p (1-\gamma) }{(1+\gamma)(1-\gamma)}
=
\frac{3^p +1 + (3^p -1)\gamma }{(1+\gamma)(1-\gamma)} 
\leq
\frac{3^p+1}{1-\gamma}
\leq
\frac{4^p}{1-\gamma} .
\end{align*}
Thus, we have
\[
\sigma \leq \frac{4}{(1-\gamma)^\frac{1}{p}} .
\]
It follows that
\begin{align*}
\dashint_{U^-(\gamma)} \lvert f_2 \rvert &\leq 
\frac{\lvert V^-(\gamma) \rvert }{\vert U^-(\gamma) \rvert} \dashint_{V^-(\gamma)} \lvert f_2 \rvert
=
\sigma^{n+p} \dashint_{V^-(\gamma)} \lvert f_2 \rvert\\
&\leq
\frac{4^{n+p}}{ (1-\gamma)^{\frac{n}{p}+1}} \dashint_{V^-(\gamma)} \lvert f_2 \rvert
\leq
C_3 M^{\gamma-}f(z)
\end{align*}
for every $z\in S^+(\gamma) \subset V^+(\gamma)$, where $C_3 = 4^{n+p}/(1-\gamma)^{\frac{n}{p}+1}$.
By taking supremum over all rectangles $U$ with $y\in U^+(\gamma)$, we obtain
$M^{\gamma-}f_2(y) \leq C_3 M^{\gamma-}f(z)$.
Since this holds for any $y\in R^-(\gamma)$ and $z\in S^+(\gamma)$, we have
\[
\dashint_{R^-(\gamma)} (M^{\gamma-}f_2)^\delta \leq
C_3 \essinf_{S^+(\gamma) \ni z} (M^{\gamma-}f(z))^\delta .
\]
By combining the estimates for $f_1$ and $f_2$, we get
\begin{align*}
\dashint_{R^-(\gamma)} (M^{\gamma-}f)^\delta &\leq 
\dashint_{R^-(\gamma)} (M^{\gamma-}f_1)^\delta + \dashint_{R^-(\gamma)} (M^{\gamma-}f_2)^\delta\\
&\leq
(C_2 + C_3)
\essinf_{S^+(\gamma) \ni z} (M^{\gamma-}f(z))^\delta .
\end{align*}
By Theorem~\ref{thm:timelagchange}, we conclude that $(M^{\gamma-}f)^\delta \in A_1^+(\gamma)$.

To prove the reverse direction,
let $w\in A_1^+(\gamma)$. 
Note that $R^-(\gamma) + (0,\gamma L^p/2) \subset R^-(\gamma/2)$.
Applying Corollary~\ref{cor:RHIlagged}, we obtain
\begin{align*}
\biggl( \dashint_{R^-(\gamma)} w^{1+\varepsilon} \biggr)^\frac{1}{1+\varepsilon}
&\leq C_1 \dashint_{R^-(\gamma)+(0,\gamma L^p/2)} w
\leq C_1 \frac{1-\gamma/2}{1-\gamma} \dashint_{R^-(\gamma/2)} w\\
&\leq \frac{C_1}{1-\gamma} \dashint_{R^-(\gamma/2)} w .
\end{align*}
By Theorem~\ref{thm:timelagchange}, we have $w \in A_1^+(\gamma/2)$.
Proposition~\ref{thm:maximalA1-cond} implies that there exists a constant $C$ such that
\begin{align*}
\bigl(M^{\gamma-}(w^{1+\varepsilon})(z) \bigr)^\frac{1}{1+\varepsilon} &=
\biggl( \sup_{R^+(\gamma) \ni z} \dashint_{R^-(\gamma)} w^{1+\varepsilon} \biggr)^\frac{1}{1+\varepsilon} \leq \frac{C_1}{1-\gamma} \sup_{R^+(\gamma) \ni z} \dashint_{R^-(\gamma/2)} w \\
&\leq \frac{C_1}{1-\gamma} \sup_{R^+(\gamma/2) \ni z} \dashint_{R^-(\gamma/2)} w
\leq
\frac{C_1}{1-\gamma} M^{\gamma/2-}w(z) 
\leq
C_2
w(z)
\end{align*}
for almost every $z\in \mathbb{R}^{n+1}$, 
where $C_2 = C_1 C / (1-\gamma)$.
Moreover, by the Lebesgue differentiation theorem~\cite[Lemma~2.3]{KinnunenMyyryYang2022}
and H\"older's inequality we have
\begin{align*}
w(z) \leq M^{\gamma-}w(z) \leq   \bigl(M^{\gamma-}(w^{1+\varepsilon})(z) \bigr)^\frac{1}{1+\varepsilon}
\end{align*}
for almost every $z\in\mathbb{R}^{n+1}$.
Hence, it follows that
\[
w(z) \leq \bigl(M^{\gamma-}f(z) \bigr)^\delta \leq C_2 w(z)
\]
for almost every $z\in\mathbb{R}^{n+1}$,
where $f=w^{1+\varepsilon}$ and $\delta=\frac{1}{1+\varepsilon}$.
Then $w=b (M^{\gamma-}f )^\delta $ almost everywhere with $b=w/(M^{\gamma-}f )^\delta$. Note that $1\leq \norm{b}_{L^\infty}\leq C_2$.
This completes the proof.
\end{proof}

Combining Theorem~\ref{thm:jones} and Theorem~\ref{thm:coifmanrochberg} we obtain the following characterization of $A_q^+$ weights.
Part $(i)$ also holds for $\gamma=0$ and $p=1$.

\begin{corollary}
Let $0<\gamma<1$.

\begin{enumerate}[(i)]
\item Assume that $f,g\in L^1_{\mathrm{loc}}(\mathbb{R}^{n+1})$ and $M^{\gamma-}f<\infty$ and $M^{\gamma+}g<\infty$ almost everywhere. Let $0<\delta<1$. Then $(M^{\gamma-}f)^\delta (M^{\gamma+}g)^{\delta(1-q)} $ is an $A_q^+(\gamma)$ weight with 
$[w]_{A_q^+(\gamma)}$
depending only on $n$ and $\delta$.

\item Assume that $w\in A_q^+(\gamma)$. Then there exist $f,g\in L^1_{\mathrm{loc}}(\mathbb{R}^{n+1})$, $0<\delta<1$ and $b$ with $b,\frac{1}{b}\in L^\infty$ such that $w=b (M^{\gamma-}f)^\delta (M^{\gamma+}g)^{\delta(1-q)} $ almost everywhere.

\end{enumerate}

\end{corollary}

\begin{proof}
By Theorem~\ref{thm:coifmanrochberg}, we have $(M^{\gamma-}f)^\delta \in A_1^+(\gamma)$ and $(M^{\gamma+}g)^\delta \in A_1^-(\gamma)$. Then Theorem~\ref{thm:jones} implies that $(M^{\gamma-}f)^\delta (M^{\gamma+}g)^{\delta(1-q)} \in A^+_q(\gamma)$.

For the reverse direction,
Theorem~\ref{thm:jones} implies that there exist $u\in A_1^+(\gamma)$ and $v\in A_1^-(\gamma)$ such that
$w=uv^{1-q}$. 
By Theorem~\ref{thm:coifmanrochberg}, there exist $f\in L^1_{\mathrm{loc}}(\mathbb{R}^{n+1})$, $0<\delta<1$ and $b_1$ with $b_1,\frac{1}{b_1}\in L^\infty$ such that 
$w=b_1 (M^{\gamma-}f)^\delta $.
Moreover, there exist $f\in L^1_{\mathrm{loc}}(\mathbb{R}^{n+1})$, $0<\delta'<1$ and $b_2$ with $b_2,\frac{1}{b_2}\in L^\infty$ such that 
$w=b_2 (M^{\gamma+}g)^{\delta'} $.
By the proof of Theorem~\ref{thm:jones}, we observe that it is possible to choose $\delta'=\delta$.
Therefore, we have $w=b (M^{\gamma-}f)^\delta (M^{\gamma+}g)^{\delta(1-q)} $ with $b = b_1 b_2^{1-q}$.
\end{proof}

\end{document}